\documentclass[11pt, leqno]{amsart}

\usepackage{amsmath,amsthm,amsfonts,amssymb,mathrsfs}
\usepackage{inputenc,mathrsfs}

\usepackage[dvips]{graphicx}
\usepackage[colorlinks]{hyperref}

\title{Ricci flow from spaces with edge type conical singularities}
\author{Lucas Lavoyer}
\email{lucas.lavoyer@uni-muenster.de}

\usepackage[margin=1.2in]{geometry}

\usepackage{setspace}
\usepackage{enumitem} 
\usepackage[style=alphabetic]{biblatex}

\usepackage{amscd}

\usepackage{frcursive}

\usepackage{xspace}

\usepackage{appendix}
\usepackage{esint}
\usepackage{calc}
\usepackage{mathtools}

\usepackage{faktor}

\newcommand{\norm}[1]{\lVert#1\rVert} 
\newcommand{\dt}{\partial_t} 
\newcommand*{\nor}{\| \hspace{-2ex}-}
\newcommand{\expander}{N\times\mathbb{R}} 
\newcommand{\gcan}{g_{\mathbb{R}}} 

\def\Xint#1{\mathchoice
{\XXint\displaystyle\textstyle{#1}}%
{\XXint\textstyle\scriptstyle{#1}}%
{\XXint\scriptstyle\scriptscriptstyle{#1}}%
{\XXint\scriptscriptstyle\scriptscriptstyle{#1}}%
\!\int}
\def\XXint#1#2#3{{\setbox0=\hbox{$#1{#2#3}{\int}$ }
\vcenter{\hbox{$#2#3$ }}\kern-.6\wd0}}

\def\dashint{\Xint-}

\newtheorem{thm}{Theorem}[section] 
\newtheorem{lem}[thm]{Lemma} 

\newtheorem*{claim}{Claim} 
\newtheorem{cor}[thm]{Corollary}

\theoremstyle{definition}
\newtheorem{defn}[thm]{Definition}

\newtheorem*{rem}{Remark}

\theoremstyle{remark}

\begin{document}


\maketitle

\begin{center}
{\it Dedicated to the memory of Prof. Richard Hamilton}
\end{center} 

\begin{abstract}
We study the Ricci flow out of spaces with edge type conical singularities along a closed, embedded curve. Under the additional assumption that for each point of the curve, our space is locally modelled on the product of a fixed positively curved cone and a line, we show existence of a solution to Ricci flow $(M,g(t))$ for $t\in (0,T],$ which converges back to the singular space as $t\searrow 0$ in the pointed Gromov--Hausdorff topology. We also prove curvature estimates for the solution and, for edge points, we show that the tangent flow at these points is a positively curved expanding Ricci soliton solution crossed with a line. 
\end{abstract}

\section{Introduction}\label{intro}
Given a smooth, closed manifold $M$ and a Riemannian metric $g_0,$ Hamilton \cite{H1} showed there is a family of metrics $g(t),$ $t\in [0,T),$ satisfying the Ricci flow equation:
\begin{align}\label{RicciFlow}
    \partial_t g(t)=-2Ric(g(t)),
\end{align}
with $g(0)=g_0.$ The existence time for this solution is bounded from below by the inverse of the maximum of the Riemannian curvature at time $t=0$ multiplied by a dimensional constant. 
 
In recent years, the question of starting Ricci flow from non-smooth spaces has been of great interest. The idea is that the regularising properties of Ricci flow  will in some sense smooth out such spaces. In general, this is not a simple task, but results have been obtained in some particular cases. In dimension two, substantial advances have already been made in, for instance, the work of Giesen--Topping \cite{Giesen_Toppin2D} showing that given any Riemannian surface, which may be incomplete and of unbounded curvature at spatial infinity, there exists a solution to Ricci flow that is instantaneously complete for $t>0.$ Later on, Topping \cite{Toppin2D} proved the solution is also unique, and, recently, Topping and Yin \cite{Topping_Yin} showed there is a solution to Ricci flow starting from Radon measures in 2D. In 3-D, Simon and Topping \cite{simon-topping2} used Ricci flow to show that any non-collapsed Ricci limit space is actually a manifold, solving a conjecture of Anderson--Cheeger--Colding--Tian. In \cite{Polyhedral-RF}, Lebedeva--Matveev--Petrunin--Shevchishin apply the results in \cite{Simon1, Simon2, Simon3} by Simon to show that any polyhedral space in three dimensions with non-negative curvature in the Alexandrov sense can be approximated by Riemannian manifolds with non-negative sectional curvature. Other papers where the Ricci flow from non-smooth initial data is studied are \cite{RF_almostNonNegativeCurvature, Der1, Hochard_paper, Huang_Tam, Lai1, Schulze_Simon, Der-Schu-Sim, Simon4, Xu1, Lamm-Simon}, where the list is definitely not exhaustive. 

In many cases, the non-smooth spaces considered might have conical points. These points can be characterised by having a Euclidean cone as their tangent cone. In \cite{conicalsing}, Gianniotis and Schulze studied the Ricci flow coming out of spaces with isolated conical singularities modelled on positively curved cones. These spaces are smooth manifolds when we exclude the conical points, and can be expected to show up as certain limits of Ricci flows developing type I singularities, see Section 1 of \cite{conicalsing}. However, spaces with non-isolated conical points can also be expected to show up as limits of manifolds; see, for instance, \cite{degenerationOfMetrics}, Example 0.29. 

The aim of this paper is to study the Ricci flow out of certain compact spaces with edge type conical singularities, generalising the results of Gianniotis--Schulze. These spaces are smooth manifolds except along a closed, embedded curve, where each point is conical. The precise meaning of this is given in Definition \eqref{singular_space}. We show short-time existence of a Ricci flow coming out of these spaces, under the additional assumption that the conical cross-section of the singular points is positively curved, and prove curvature bounds for the solution. 

\textbf{Main statement.} Before stating the main result of the paper, we define what it means for a space to be a compact space with edge type conical singularities. In order to do so, we recall the definition of a metric cone over a smooth manifold.
\begin{defn}
Let $(X,g_X)$ be a smooth, closed $(n-1)-$dimensional Riemannian manifold. The topological cone over $X,$ $C(X),$ is defined as the set of equivalences in the product $[0,\infty]\times X:$
\begin{align*}
    C(X)=\{ [0,\infty] \times X\}/ \sim ,
\end{align*}
where the equivalence relation is given by $(0,x)\sim(0,y),$ for every $x,y\in X.$ We denote the vertex of $C(X)$ by $o.$ For $R>0,$ let
\begin{align*}
    C_R(X) := \{ (r,x)\in C(X) \hspace{0.1cm}|\hspace{0.2cm} r < R \}.
\end{align*}
We also equip $C(X)$ with the standard conical metric $g_c = dr^2 + r^2g_X.$
\end{defn}

\begin{defn}\label{singular_space} 
The pair $(Z,g_Z)$ is a compact space with edge type conical singularities modelled on $\displaystyle{(C(X)\times\mathbb{R},G=g_c\otimes\gcan)}$ along a closed, embedded curve $\Gamma\subset M$ 
if 
\begin{enumerate}
    \item $(Z\backslash \Gamma, g_Z)$ is a smooth Riemannian manifold, with $g_Z$ an incomplete metric.
    
    \item $(Z,d_Z)$ is a compact metric space, where $d_Z=d(g_Z)$ is the metric induced by $g_Z.$
    
    \item There exist uniform $r_0 >0$ and $\eta_0>0$ such that $\forall p\in \Gamma,$ there exists a map  
    \begin{align*}
        \phi_p: (0,r_0]\times X \times [-\eta_0/2,\eta_0/2] \longrightarrow  Z
    \end{align*}
   such that $\phi$ is a diffeomorphism onto its image,
   \begin{align*}
       \lim_{r\to 0}\phi_p(r,x,0)=p\in \Gamma,
   \end{align*}
and the length of $\Gamma$ from $q^-$ to $q^+,$ where $\displaystyle{\lim_{r\to 0}\phi_p(r,x,\pm \eta_0/2)=q^{\pm}}$ for any fixed $x\in X,$ is such that  $\mathbf{l}(\Gamma_{|_{[q^-,q^+]}})=\eta_0.$ 

    \item Finally, for every $p\in \Gamma$ and $0<r<r_0,$ it holds that
    \begin{align}\label{blow up to cone}
      \sum_{j=0}^{4}r^j\left| (\nabla^G)^j(\phi_{p}^*g_Z - G)\right|{\big|_{B_r(p)}} < \kappa(r),
    \end{align}
 where $\kappa(r)\to 0$ as $r\to 0.$ 
\end{enumerate}

\end{defn}

Given the definition above, we see that $C(X)\times \mathbb{R}$ with the product metric is the Gromov--Hausdorff tangent cone at any $p\in \Gamma.$ When $X$ is a closed, simply connected manifold with $Rm(g_X) \geq 1,$ we remark that Deruelle \cite{Der1} showed there is a unique expanding Ricci soliton asymptotic to $C(X).$ From now on, we will assume the conditions above on $X.$

\begin{thm}\label{mainTheorem}
Let $\displaystyle{(Z,g_Z)}$ be a compact space with edge type conical singularities along a closed, embedded curve $\Gamma,$ each modelled on
\begin{align*}
    \left( C(\mathbb{S}^{n-1})\times\mathbb{R}, G= dr^2 + r^2g+dl^2\right),
\end{align*}
with $\displaystyle{Rm(g)\geq 1},$ but $\displaystyle{Rm(g)\not\equiv 1}.$ Under these assumptions, there exists a Ricci flow solution $(M,g(t))_{t\in (0,T]},$ where $M$ is a smooth manifold, and a constant $C_M$ satisfying the following.
\begin{enumerate}
    \item $(M,d_t) \rightarrow (Z,d_Z)$ in the Gromov--Hausdorff topology as $t\searrow 0,$
    \item there exists a map \begin{align*}
        \Psi: Z\backslash \Gamma \longrightarrow M,
    \end{align*}
a diffeomorphism onto its image, such that $\displaystyle{\Psi^*g(t)\to g_Z}$ in the $C^{\infty}$ topology as $t\searrow 0,$ uniformly away from $\Gamma,$
    \item $\displaystyle{|Rm(g(t))|_{g(t)}\leq \frac{C_M}{t}}$ for $t\in (0,T],$
    \item given any sequence $t_k \searrow 0$ and $p_k\in Im(\Psi)^c,$ with $p_k \to q\in \Gamma$ under the Gromov--Hausdorff convergence from above, we have:
    \begin{align*}
        \left( M, \frac{1}{t_k}g(t_kt),p_k\right)_{t\in (0,{t_k}^{-1}T]} \rightarrow \left( \expander, {g_{\exp}}(t),\Bar{q}\right)_{t\in (0,\infty)},
    \end{align*}
    where $\displaystyle{\left( \expander, {g_{\exp}}(t),\Bar{q}\right)_{t\in (0,\infty)}}$ is the product Ricci flow induced by $(\expander, g_N\otimes\gcan),$ and $(N,g_N)$ is the unique expander with positive curvature operator asymptotic to $\displaystyle{( C(\mathbb{S}^{n-1}), g_c)}.$
\end{enumerate}

\end{thm}

\begin{rem}
We remark that the theorem still holds, with only minor modifications to the proof, if the metric $g$ on the link of the cones, with $Rm(g)\geq 1,$ varies smoothly along $\Gamma.$ 
\end{rem}

Our approach to proving the result above is the following. We start by desingularising the initial metric. The idea is to approximate, for each point $p\in \Gamma$ on the curve, a small neighbourhood of $p$ by $N\times[-\eta,\eta],$ where $\eta>0$ is a small number that will depend on $\Gamma,$ and the expander $N$ is glued in at a small scale $s>0.$ We cover $\Gamma$ in this way, making sure our glueing is smooth so that we obtain a smooth Riemannian manifold $(M_s,g_s)$ that approximates our initial space. The expanding soliton $(N,g_N(t))$ used above is known to be stable due to the work of Deruelle--Lamm \cite{Der-Lamm}. We extend their result to $\expander$ with the product Ricci flow solution. After this, we prove uniform curvature bounds for the approximations via an iteration argument. The idea is that on a small neighbourhood of any $p\in \Gamma,$ the approximated solution will stay close to the product Ricci flow on $\expander,$ in an appropriate sense, assuming bounds on the boundary of this neighbourhood. Since we do not want to assume any control over time, we use pseudolocality to control the solution far enough from the curve (where the metric is conical), and an iteration argument with overlapping neighbourhoods as above to show we do not need any control on the horizontal ends of this neighbourhood. Then we let $s\searrow 0$ to obtain a limit solution. 

We observe that the condition on the curvature of the model cones is only used when we apply the existence and stability of such expanding solitons; our solution does not depend on any curvature bounds for the initial data or the approximating solutions. In particular, this construction could be carried out in the same way if existence and stability of expanders out of more general cones were proven.

\textbf{Outline.} We now give an overview of how the paper is organised. In Section \ref{preliminaries}, we recall a few properties of expanding Ricci solitons asymptotic to cones and define the model class we will be working with. In Section \ref{stabilitySection} we adapt the stability result of Deruelle--Lamm \cite{Der-Lamm} to extend it to $\expander$ with the product metric, which can be seen as a nonnegatively curved expanding soliton. The observation is that the curvature is still decaying with the inverse of the distance to the spine $\displaystyle{\{q_{\max}\} \times \mathbb{R}},$ where $q_{\max}$ is the unique critical point of the potential function $f_N.$ Together with the splitting of the heat kernel on a product manifold, this turns out to be enough to obtain the same stability result.

In Section \ref{uniformEstimates} we localise the stability obtained in Section \ref{stabilitySection} and prove uniform estimates for solutions in our model class. In order to do that, we split our space into two regions: the \textit{conical} and \textit{expanding} region. On the conical region, our metric is almost conical and, therefore, almost flat. This allows Perelman's pseudolocality to control the flow. We are then able to obtain uniform control along the curve using an overlap argument, moving from a local control, i.e., of a small neighbourhood of a (any) point in the approximate curve $\Gamma_s$ to a result controlling a neighbourhood of the whole curve. 

The observation is that the estimates on the conical region control our solution far out in the neighbourhood, but we also initially assume control on the 'horizontal ends' of such neighbourhood, which puts us in a position where we can apply the local stability result to obtain estimates on a smaller neighbourhood of a point in the curve. In order to do so, we rely on the fact that the heat kernel $K_L$ associated to the Lichnerowicz operator (defined in section \ref{stabilitySection}) splits on a product manifold, so $K_L=K^N_{L}K^{\mathbb{R}},$ where $K^{N}_{L}$ and $K^{\mathbb{R}}$ are the heat kernels on $N$ and $\mathbb{R},$ respectively. Along with this, the classical decay estimate for the heat kernel in Euclidean space combined with the results in \cite{Der-Lamm} yield the needed estimate:
\begin{align*}
    \| K_L(x,t,y,s)\| \leq \frac{c}{(t-s)^{\frac{n+1}{2}}}\exp\left\{ -\frac{d^2_{g_0(s)}(x,y)}{D(t-s)}\right\},
\end{align*}
for $0\leq s < t$ and $x,y \in N\times\mathbb{R},$ where $c$ will depend only on the expander, $g_0(t)=g_N(t)\otimes \gcan(t)$ is the product Ricci flow induced by $(\expander,g_N\otimes \gcan),$ and the norm is also defined in section \ref{stabilitySection}. The exponential decay of the heat kernel in the $\mathbb{R}-$direction implies that the influence of the 'horizontal ends' is small at the centre. We do this for enough points $p_k \in \Gamma_s$ so that we cover $\Gamma_s.$ This can then be improved indefinitely by overlapping such regions until we obtain that the control only depends on estimates in the conical region. 

Finally, in Section \ref{theSolution} we construct the approximating solution explicitly and pass to a limit with $s\searrow 0.$ This gives us a solution to the Ricci flow that exists for $t\in (0,T]$ and has the right curvature decay. Furthermore, as $t\searrow 0,$ we show that away from $\Gamma,$ the Ricci flow $g(t)$ converges to $g_Z$ locally smoothly uniformly after being pulled back by an appropriate map. The solution also converges back to $(Z,d_Z)$ in the Gromov--Hausdorff sense. 

\subsection{Acknowledgements} The author thanks his supervisor, Felix Schulze, for helpful discussions and encouragement. The author was supported by the UK Engineering and Physical Sciences Research Council (Grant number: EP/V520226/1).

\section{Preliminaries}\label{preliminaries} 

We start by briefly recalling the definition of gradient Ricci solitons, first introduced by Hamilton in \cite{HAM2D}. The triple $(N,g_N,f_N),$ where $(N,g_N)$ is a Riemannian manifold and $f_N$ is a smooth function on $N,$ is a gradient Ricci soliton if it satisfies 
\begin{align*}
   Hess_{g_N}f_N= \frac{1}{2}\mathcal{L}_{\nabla f_N} g_N=Ric(g_N)+ \lambda{g_N},
\end{align*}
for some constant $\lambda,$ where $\mathcal{L}$ denotes the Lie derivative. The soliton is expanding if $\lambda>0,$ steady if $\lambda=0$ or shrinking if $\lambda<0.$ In our case, $\lambda >0$ and the soliton is an \textbf{expanding} gradient Ricci soliton. We can always normalise the metric so that $\displaystyle{\lambda= \frac{1}{2}.}$ If the soliton has bounded curvature, by possibly changing the potential function by a constant if necessary, we can assume that $f_N\geq 0.$ An expanding Ricci soliton induces a solution to the Ricci flow given by 
\begin{align}
    g_N(t)=  t(\phi_N^t)^*g_N, \hspace{1cm} \text{ for } t>0,
\end{align}
where $\phi_N^t$ is the family of diffeomorphisms generated by $\displaystyle{-\frac{1}{t}\nabla f_N}$ with $\phi_N^1=id.$

We can now give a precise definition of an expanding Ricci soliton coming out of a cone. 

\begin{defn}\label{asymptotic expander}
Let $(X,g_X)$ be a smooth Riemannian manifold and 
\begin{align*}
    (C(X),g_c=dr^2 + r^2g_X,o)
\end{align*}
be the associated cone with vertex $o.$ We say that the expanding Ricci soliton $(N,g_N,f_N)$ is asymptotic to the cone $C(X)$ if
\begin{enumerate}
    \item for some $\Lambda_0>0,$ there is $\displaystyle{F:[\Lambda_0,\infty)\times X \longrightarrow N},$ a diffeomorphism onto its image,  such that $\displaystyle{N\backslash Im(F)}$ is compact and 
    \begin{align*}
        f_N(F(r,x))=\frac{r^2}{4},
    \end{align*}
    for every $(r,x)\in [\Lambda_0,\infty)\times X.$
    \item \begin{align*}
        \sum_{j=0}^{4}\sup_{\partial B(o,r)}r^j|(\nabla^{g_c})^j(F^*g_N -g_c)|_{g_c}=k_{e}(r),
    \end{align*}
    where $k_e(r)\to 0$ when $r\to \infty.$
\end{enumerate}
\end{defn}
In \cite{Der1}, Deruelle shows that if $X$ is diffeomorphic to the standard sphere $\displaystyle{\mathbb{S}^{n-1}}$ and the metric on the sphere satisfies $\displaystyle{Rm(g)\geq 1},$ but $\displaystyle{Rm(g)\not\equiv 1},$ then the expanding soliton smoothing out the cone over $X$ exists and it is unique. 

For the expanding soliton, an important observation is that the notion of asymptotic cone coincides with the notion of rough initial data in the Gromov--Hausdorff sense (see \cite{Der1} for the details). It is also standard to consider the following natural coordinate at infinity on the expander:
\begin{align*}
    \mathbf{r}:= 2\sqrt{f}=(F^{-1})^*r.
\end{align*}
We use $\mathbf{r}$ to measure the distance from a point on the manifold to the tip of the expander, i.e., the critical point of $f_N,$ which is unique when $Rm(g_N)\geq 0$ and the expander is normalised. In fact, since we will be working with the expander at scale $s>0,$ we also define the radial coordinate for this case, given by $\mathbf{r}_s=2\sqrt{sf_s},$ where $f_s=f\circ \phi_s.$ We can also define ${F_s: [\Lambda_0 \sqrt{s},\infty)\times X\rightarrow N}$ by $\displaystyle{F_s=(\phi_s)^{-1}\circ F \circ a_s,}$ where $\displaystyle{a_s(r,x)=\left(\frac{r}{\sqrt{s}},x\right)}$ for $(r,x)\in [0,\infty)\times X.$ Hence, $\mathbf{r}_s(F_s(r,x))=r.$ 

The next lemma shows us that the expander at scale $s$ converges to the cone as $s\to 0$ and it is indeed natural to consider $F_s.$ The proof is a straightforward computation (see equation 2.6 on \cite{conicalsing}).

\begin{lem}\label{expander_to_cone}
Under the construction above, $\displaystyle{F_{s}^{*}g_N(s)}$ converges to $g_c$ as $s\to 0,$ uniformly away from $o$ in $\displaystyle{C^{4}_{loc}}.$
\end{lem}

We also state a second lemma from \cite{conicalsing}, which will be useful to us later.
\begin{lem}\label{lemma21anal}(Lemma 2.1, \cite{conicalsing}).
Let $\displaystyle{(N, g_N, f_N)}$ be an asymptotically conical gradient Ricci expander, and let $g_N(t),$ $t\geq 0,$ be the induced Ricci flow with $g_N(0)=g_N.$ There exists $\gamma_0\geq 1,$ $c,\Lambda_0>0$ such that 
\begin{align*}
   & |{F}^*g_N - g_c|_{g_c} +\mathbf{r}|\nabla^{g_c}{F}^*g_N|_{g_c} < \frac{1}{100}, \hspace{0.3cm} \frac{1}{2}\leq |\nabla^{g_N}\mathbf{r}|_{g_N}\leq 2\\
    & |\mathbf{r}\Delta_{g_N}\mathbf{r}|_{g_N}\leq 4(n-1), \hspace{0.3cm} \mathbf{r}^2|Rm(g_N)|_{g_N}\leq C(g_c)
\end{align*}
on $\displaystyle{\left\{ (x,t)\in N\times [0,+\infty); \hspace{0.1cm} \mathbf{r}(x)\geq \sqrt{\gamma_0 t +\Lambda_0^2} \right\}}.$
\end{lem}

From now on, $(N,g_N, f_N)$ will be a positively curved expanding gradient Ricci soliton, $q_{\max}$ the unique maximal point of $f_N,$ and $g_N(t)$ will be its associated Ricci flow. We also observe that given the Euclidean metric on $\mathbb{R},$ $g_{\mathbb{R}},$ one can construct a static solution to the Ricci flow which has an expanding structure. Define the potential function on $\mathbb{R}$ by $f_{\mathbb{R}}(l):= \frac{l^2}{4}$ and let $\displaystyle{g_{\mathbb{R}}(t)=t(\phi_{\mathbb{R}}^t)^*g_{\mathbb{R}}},$ where $\displaystyle{ \phi_{\mathbb{R}}^t:\mathbb{R}\longrightarrow \mathbb{R}}$ is a family of diffeomorphisms with $t > 0,$ $\phi_{\mathbb{R}}(\cdot,1)=id_{\mathbb{R}}$ and 
\begin{align*}
    \dt \phi_{\mathbb{R}}^t=-\frac{1}{t}\nabla f_{\mathbb{R}}(\phi^{t}_{\mathbb{R}}).
\end{align*}
Integrating this, $\phi_{\mathbb{R}}^t$ is then given by $\displaystyle{ \phi_{\mathbb{R}}^t=\frac{1}{\sqrt{t}}id_{\mathbb{R}}.}$ Thus, $g_{\mathbb{R}}(t)\equiv g_{\mathbb{R}}$ for all times and it solves the Ricci flow equation with $Ric(g_{\mathbb{R}}(t))\equiv 0.$ 

Given the construction above, we can consider the expander $\displaystyle{ (N\times \mathbb{R},g_{\exp} ,{f}_0 ),}$ where ${g}_{\exp}=g_N\otimes g_{\mathbb{R}},$ ${f}_0(x,l)= f_N(x) + f_{\mathbb{R}}(l),$ and the associated Ricci flow given by $\displaystyle{{g_{\exp}}(t)=g_N(t)\otimes g_{\mathbb{R}}(t)}.$ Finally, let $\displaystyle{\tilde{F}_s(r,x,l)=(F_s(r,x),l)}$ for every $(r,x,l)\in  [\Lambda_0 \sqrt{s},\infty)\times X \times \mathbb{R}.$

We now define the model class $\mathcal{M}(\delta,\Lambda,s)$ that we will be working with. Metrics in $\mathcal{M}(\delta,\Lambda,s)$ can be seen as the smoothing of an edge type conical singularity with an expander at scale $s$ cross an interval of the real line. In Section \ref{theSolution}, our approximating solution will be constructed as an element of $\mathcal{M}(\delta,\Lambda,s)$ for appropriate values of $\delta,\Lambda$ and $s.$

\begin{defn}\label{aproxclass}
Given $\delta, s >0,$ $\Lambda \geq \Lambda_0,$ we say that the pair $(M,g)$ belongs to $\mathcal{M}(\delta, \Lambda,s)$ if it is a complete Riemannian manifold with bounded curvature satisfying the following. There exist a closed, embedded curve $\Gamma_s \subset M,$ with length $\mathbf{l}(\Gamma_s)=L>0,$ and a map $\Phi_s: \{\mathbf{r_s}\leq 1\}\times [-L/2,L/2] \longrightarrow M$ such that $\displaystyle{\Phi_s}$ is a diffeomorphism onto its image, and $\Phi_s\left(\{q_{\max}\}\times [-L/2,L/2] \right)=\Gamma_s,$  where $q_{\max}\in N$ is the unique critical point of the potential function $f_N.$ Furthermore, there exist $0<\eta_0 < L$ and functions $r_s:Im(\Phi_s)\longrightarrow [\Lambda\sqrt{s},1],$ defined by 
    \begin{align*}
        r_s=max\left\{ (\pi_1\circ \Phi_{s}^{-1})^*\mathbf{r_s},\Lambda\sqrt{s}\right\},
    \end{align*}
where $\pi_1: N\times \mathbb{R} \longrightarrow N$ is the natural projection, and $\displaystyle{l_s: Im(\Phi_s)\longrightarrow \mathbb{R} }$ given by
\begin{align*}
    l_s= \pi_2 \circ \Phi_s^{-1}, 
\end{align*}
where $\pi_2: N\times \mathbb{R} \longrightarrow \mathbb{R}$ is the projection on the second coordinate, such that for all $p\in Im(\Phi_s),$ the following holds. If $\Phi_s(x,l)=p,$ then
\begin{align*}
    {\Phi_s\left( \{q_{\max}\}\times[l -\eta_0/2,l +\eta_0/2]\right)\subset \Gamma_s}
\end{align*}
is isometric to $[-\eta_0/2,\eta_0/2]\subset \mathbb{R}$ and
\begin{align}\label{closeToCone}
    \sum_{j=0}^{4}r^j\left| \left(\nabla^G\right)^j\left((\Phi_s\circ \tilde{F}_s)^* g- G\right)\right|_G +r^j\left|(\nabla^G)^j(\tilde{F}_{s}^{*}g_{\exp}(s)-G)\right|_G < \delta
\end{align}
in $[\Lambda\sqrt{s},1]\times  X \times [-\eta_0/2,\eta_0/2],$ where $G=g_c\otimes\gcan$ is the product metric on $C(\mathbb{S}^{n-1})\times\mathbb{R}$ with the same assumptions as in Theorem \ref{mainTheorem}, and
\begin{align}\label{closeToExp}
    \left|\Phi_{s}^{*}g-g_{\exp}(s)\right|_{g_{\exp}(s)} <\delta
\end{align}
in $\{\mathbf{r_s} \leq 2(\Lambda+1)\sqrt{s}\} \times [-\eta_0/2,\eta_0/2].$ 
\end{defn}

\section{Stability Under Ricci Flow}\label{stabilitySection}
This section is dedicated to proving stability of $
\expander.$ We generalise, to our particular case, the weak stability for expanders with positive curvature operator of Deruelle--Lamm \cite{Der-Lamm}. We follow \cite{Der-Lamm} as much as we can, pointing out when new ideas are needed. We observe that, for consistency with the literature, in this section we have $\displaystyle{{g}_0(t)=g_N(t)\otimes g_{\mathbb{R}}(t)},$ $\displaystyle{g_N(t)=(1+t)(\phi_N^t)^*g_N}$ and $\displaystyle{g_{\mathbb{R}}(t)=(1+t)(\phi_{\mathbb{R}}^t)^*g_{\mathbb{R}}},$ where $\phi^t_N$ is generated by $-\frac{1}{1+t}\nabla f_N$ with $\phi^0_N=id_N$ and the definition of $\phi_{\mathbb{R}}^t$ is analogous. Thus, ${g_{\exp}}(1+t)=g_0(t).$
\subsection{The set up}

Let $g(t)$ be a solution to Ricci flow on $N\times \mathbb{R}$ for $t\geq 0$ such that 
\begin{align}
    g(0)={g}_0(0)+h,
\end{align}
where $h \in S^2T^*M$ is such that $g(0)$ is a metric. We study the following associated problem (see section 2 of \cite{Der-Lamm} for the details). Let $\displaystyle{h(t):= g(t)-{g}_0(t) }$ and consider $\Bar{h}(t)=(1+t)(\tilde{\phi}_t)^*h(\ln(1+t)),$ where $\tilde{\phi}_t=(\phi_N^t,\phi_{\mathbb{R}}^t).$ Then the evolution of $h(t)$ can be written as
\begin{align}\label{MRDF}
    (\partial_t - L_t)\Bar{h}=R_0[\Bar{h}]+\nabla^{g_0(t)}R_1[\Bar{h}],
\end{align}
with 
\begin{align*}
    R_0[\Bar{h}]:= \Bar{h}^{-1}*\Bar{h}*Rm(g_0(t))+ (g(t))^{-1}*(g(t))^{-1}*\nabla^{g_0(t)}\Bar{h}(t)*\nabla^{g_0(t)}\Bar{h}(t),
\end{align*}
\begin{align*}
    \nabla^{g_0(t)}R_1[\Bar{h}]_{ij}:= \nabla_{k}^{g_0(t)}\left(\left( (g_0(t)+\Bar{h})^{kl}-g_0(t)^{kl}\right)\nabla^{g_0(t)}_{l}\Bar{h}_{ij}\right),
\end{align*}
where $L_t$ is the time dependent Lichnerowicz operator given by
\begin{align*}
     L_t\Bar{h}=\Delta_{g_0(t)}\Bar{h}+2Rm(g_0(t))*\Bar{h}-Ric(g_0(t))\otimes \Bar{h} - \Bar{h}\otimes Ric(g_0(t)), \hspace{0.5cm} t\geq 0. 
\end{align*}

We can also express the flow \eqref{MRDF} globally by:
\begin{align*}
    &\partial_t\Bar{g}(t)=-2Ric(\Bar{g}(t))+ \mathcal{L}_{V(\Bar{g}(t),{g}_0(t))}(\Bar{g}(t)),\\
    & \Bar{g}(t):= {g}_0(t)+\Bar{h}(t),
\end{align*}
where $\displaystyle{V(g(t),{g}_0(t)):= div_{g(t)}(g(t)-{g}_0(t)) -\frac{1}{2}\nabla^{g(t)}tr_{g(t)}(g(t)-{g}_0(t)) }.$

We define function spaces $X$ and $Y$ as in \cite{Koch-Lamm}. First, for $p\in (0,\infty]$ and a family of tensors $(h(t))_{t\geq 0},$ we define the average parabolic $L^p-$norm of $h$ as follows. Let $(\expander,g(t))_{t\geq 0}$ be a Ricci flow. Then

\begin{align*}
     \nor h \|_{L^p(C(x,0,R))}:= \left( \dashint_{C(x,0,R)} |h|^p_{g_0(s)}(y,s)d\mu_{g_0(s)}(y)ds\right)^{\frac{1}{p}}
\end{align*}
and
\begin{align*}
     C(x,0,R):= \{ (y,l,s)\in N\times\mathbb{R}\times \mathbb{R}^*_+ |\hspace{0.1cm} s\in (0,R^2],\hspace{0.1cm} & y\in B_{{g_N}(s)}(x,R),\\
     &l\in (-R,R) \}.
\end{align*}
We then define the function space $X$ as the completion of $\displaystyle{\{h \hspace{0.1cm}|\hspace{0.1cm}\|h\|_{X}<+\infty\}}$ under its norm, where
\begin{align*}
\|h\|_X&:=\sup_{t\geq 0}\|h(t)\|_{L^{\infty}(N\times \mathbb{R},{g_0}(t))} \\
    &+\sup_{N\times\mathbb{R}\times \mathbb{R}_+^*}\left(R\nor  \nabla h\|_{L^2(C(x,l,R))}+\nor\sqrt{t}\nabla h\|_{L^{n+4}\left(C(x,l,R)\setminus C\left(x,l,\frac{R}{\sqrt{2}}\right)\right)}\right),
\end{align*}
and, analogously, the space $Y$ as 
\begin{align*}
    Y:= Y_0 + \nabla Y_1
    := \{ &(R_0(t) + \nabla^{{g}_0(t)}_i R^i_1(t))_{t\geq 0} |
    (R_0(t))_{t\geq 0}\subset  S^2T^*(N\times\mathbb{R}) ;\\
    &(R_1(t))_{t\geq 0 }\subset S^2T^*(N\times\mathbb{R}) \otimes \Gamma(T(N\times\mathbb{R}))\},
\end{align*}
where the norm is given by 
\begin{align*}
\|R_0\|_{Y_0}:=\sup_{(x,l,R)\in N\times\mathbb{R}\times\mathbb{R}_+^*}\bigg(&R^2\nor R_0\|_{L^1( C(x,l,R))}\\
    &+R^2\nor R_0\|_{L^{\frac{n+4}{2}}\left(C(x,l,R)\setminus C\left(x,l,\frac{R}{2}\right)\right)}\bigg),
\end{align*}
and
\begin{align*}
    \|R_1\|_{Y_1}:=\sup_{(x,l,R)\in N\times\mathbb{R}\times\mathbb{R}_+^*}\bigg(&R\nor R_1\|_{L^2(C(x,l,R))}\\
    &+\nor \sqrt{t}R_1\|_{L^{n+4}\left(C(x,l,R)\setminus C\left(x,l,\frac{R}{2}\right)\right)}\bigg).
\end{align*}

The following is a minor extension of \cite[Lemma 3.1]{Der-Lamm}. For completeness, we provide their proof below, adding the small modifications needed to extend the result to $\expander.$ 

\begin{lem}\label{lemma31DerLamm}
Let $(N^n, g_N,f_N),$ $n\geq 3,$ be an expanding gradient Ricci soliton with quadratic curvature decay. Consider $(\expander,g_0,f_0)$ as in Section \ref{preliminaries}. Then, for any $\gamma \in (0,1),$ the operator $R_0[\cdot] + \nabla R_1[\cdot]: B_X(0,\gamma)\subset X \to Y $ is analytic and satisfies 
\begin{align*}
\|R_0[h]+\nabla R_1[h]\|_Y \leq c(n,\gamma,g_0)\|h\|_X^2, 
\end{align*}
and
\begin{align*}
\|R_0[h']-R_0[h]+\nabla(R_1[h']-R_1[h])\|_Y\leq c(\|h'\|_X + \|h\|_X)\|h'-h\|_X,
\end{align*}
for any $h,h' \in B_X(0,\gamma),$ where $c=c(n,\gamma,g_0).$

\end{lem}

\begin{proof}
Most of the estimates can be directly checked. The only non-trivial estimate is the zeroth order quadratic term of $R_0[h],$ which relies on the decay of curvature tensor of $g_N$. First, observe that the metric $g_N(t)$ also has quadratic curvature decay, and if $p\in N$ is such that $\nabla^g_N f_N(p)=0,$ then
\begin{equation*}
    |Rm(g_N(t))|_{g_N(t)}(x)\leq \frac{c}{1+t+d_{g_N(t)}^2(p,x) }
\end{equation*}
for $t\geq 0,$ $x\in N$ and $c>0$ independent of time. Therefore,
\begin{align*}
    \dashint_{B_{g_N(t)}(x,R)}|Rm(g_N(t))|d\mu_{g_N(t)}\leq c\dashint_{B_{g_N(t)}(x,R)}\frac{1}{1+t+d_{g_N(t)}^2(p,y)}d\mu_{g_N(t)},
\end{align*}
for any non-negative $t.$ If $d_{g_N(t)}(p,x)\geq 2R,$ then $d_{g_N(t)}(p,y)\geq R$ for any $y\in B_{g_N(t)}(x,R)$ and 
\begin{align*}
    \dashint_{B_{g_N(t)}(x,R)}\frac{1}{1+t+d_{g_N(t)}^2(p,y)}d\mu_{g_N(t)} \leq \frac{c}{1+R^2,}
\end{align*}
for a positive constant $c,$ uniform in time, space and radius $R.$ If $d_{g_N(t)}(p,x)\leq 2R,$ then the co-area formula yields
\begin{align*}
    \int_{B_{g_N(t)}(x,R)} & \frac{1}{1+t+d_{g_N(t)}^2(p,y)}d\mu_{g_N(t)}\leq \int_{B_{g_N(t)}(x,3R)}\frac{1}{1+t+d_{g_N(t)}^2(p,y)}d\mu_{g_N(t)}\\
    &\leq c(n)\int_{0}^{3R}\frac{r^{n-1}}{1+r^2}dr\leq c(n)R^{n-2},
\end{align*}
if $n\geq 3.$ In any case,
\begin{align*}
    \int_{0}^{R^2}\dashint_{B_{g_N(t)}(x,R)}|Rm(g_N(t))|d\mu_{g_N(t)}dt \leq c,
\end{align*}
which is enough to prove the result in \cite{Der-Lamm}. 

We now show the same bounds for $N\times\mathbb{R}$ with $g_0(t)=g_N(t)\otimes\gcan(t).$ First, we have $\displaystyle{|Rm({g}_0(t))|_{{g}_0(t)}=|Rm({g}_N(t))|_{g_N(t)}},$ since  $\displaystyle{Rm({g}_0(t))(\partial_r,\cdot,\cdot,\cdot)\equiv 0}.$ The integral over a ball $B_{g_0(s)}(\tilde{x},R)$ in $\expander$ can be estimated by an integral over a cylinder $B_{g_N(s)}(x,2R)\times (-2R,2R),$ where we are assuming, without loss of generality, that $\tilde{x}=(x,0).$ Therefore, using the result above, we have
\begin{align*}
&\int_{0}^{R^2}\dashint_{C_{{g}_0(t)(\tilde{x},2R)}}|Rm({g}_0(t))|d\mu_{{g}_0(t)}dt \\
&\leq c(n)\int_{0}^{R^2}\dashint_{-2R}^{2R}\dashint_{B_{g_N(t)({x},2R)}}|Rm({g}_0(t))|d\mu_{{g}_N(t)}drdt\\
& =\dashint_{-2R}^{2R}\int_{0}^{R^2}\dashint_{B_{g_N(t)({x},2R)}}|Rm({g}_N(t))|d\mu_{{g}_N(t)}dtdr\\
&\leq c(n)\frac{1}{R}\int_{-2R}^{2R}c(n)dr = c(n).
\end{align*}
\end{proof}

\subsection{The homogeneous case}

When studying equation \eqref{MRDF}, the first step is to estimate the solution to the homogeneous linear equation 
\begin{align}\label{homogeneous}
    \dt h=L_t h, \hspace{0.5cm} h(0)=h_0 \in L^\infty(S^2T^*(N\times \mathbb{R})),
\end{align}
for $t\geq 0.$ In \cite[Theorem 4.1]{Der-Lamm}, the authors show that for a positively curved expander $(N,g_N(t))$ with quadratic curvature decay, we get $\displaystyle{\| h\|_X \leq c \|h_0\|_{L^\infty(N,g_N)}.}$ In our case, $(N\times \mathbb{R}, {g}_0(t))$ only has non-negative curvature operator. Furthermore, the curvature decay does not hold along the $\mathbb{R}-$direction. However, just like we did for Lemma \ref{lemma31DerLamm}, we can still prove similar estimates. We state this below and provide details of the parts where our proof is different from \cite{Der-Lamm}.

\begin{thm}\label{homogEst}
Let $(N^n \times \mathbb{R},{g}_0(t))_{t\geq 0},$ $n\geq 3,$ be an expanding gradient Ricci soliton with the usual product metric. Assume that $(N^n,g_N(t))$ has positive curvature operator and quadratic curvature decay. Let $(h(t))_{t\geq 0}$ be a solution to the homogeneous linear equation \eqref{homogeneous}. Then $h(t)\in X$ for every $t\geq 0$ and $\displaystyle{\| h\|_X \leq c \|h_0\|_{L^\infty(N\times\mathbb{R},{g}_0)}. }$
\end{thm}

\begin{proof}
Given the norm defined for $X,$ we need an $L^\infty$ estimate, an $L^2$ and an $L^{n+4}$ estimate. The $L^\infty$ estimate follows directly from the fact that the Ricci curvature is non-negative for the expander $N\times \mathbb{R}.$ Similarly, for the $L^{n+4}$ estimate to work we only need the extra facts that the solution is of Type III, i.e.,
\begin{align*}
    |Rm(g_0(t))|_{g_0(t)} + \sqrt{t}|\nabla^{g_0(t)}Rm(g_0(t))|_{g_0(t)}\leq \frac{c}{1+t},
\end{align*}
and the Ricci flow is non-collapsed. All of these still hold when we take the product of the expander with $\mathbb{R}$ and consider the product metric ${g}_0(t)=g_N(t)\otimes g_{\mathbb{R}}(t).$

For the $L^2$ estimate, we need to be more careful. Our goal is to show 
\begin{align*}
    R \nor \nabla h\|_{L^2(C(x,l,R))}\leq c \|h_0\|_{L^{\infty}(\expander)},
\end{align*}
for arbitrary $(x,l,R)\in \expander\times \mathbb{R}_+^*.$ This estimate heavily relies on the quadratic decay of the curvature. Computing the evolution equation for $|h|^2,$ where the norm is always with respect to ${g}_0(t),$ we get
\begin{align*}
    \dt |h|^2 \leq \Delta_{g_0(t)} |h|^2 -2 |\nabla^{g_0(t)}h|^2 +c(n+1)|Rm(g_0(t))\|h|^2.
\end{align*}
We multiply this inequality by a smooth cut-off function $\displaystyle{ \psi^2:N\times\mathbb{R}\times\mathbb{R}_+ \longrightarrow \mathbb{R}_+}$ and integrate by parts in space to get
\begin{align*}
    & \dt \int_{N\times\mathbb{R}} \psi^2|h|^2d\mu_{g_0(t)} + 2\int_{N\times \mathbb{R}} \psi^2|\nabla^{g_0(t)}h|^2d\mu_{g_0(t)}  \\
   & \leq  \int_{N\times\mathbb{R}} \left[ \left< -\nabla^{g_0(t)}\psi^2, \nabla^{g_0(t)}|h|^2\right>
     +\left(\dt \psi^2 +c(n)\psi^2|Rm(g_0(t))|\right)|h|^2\right]d\mu_{g_0(t)}\\
   &  \leq \int_{N\times\mathbb{R}}\left[ 4\psi|\nabla^{g_0(t)}\psi\|h\|\nabla^{g_0(t)}h|+\left(\dt \psi^2 +c(n)\psi^2|Rm(g_0(t))|\right)|h|^2\right]d\mu_{g_0(t)}. 
\end{align*}

Young's inequality applied to $4\psi|\nabla^{g_0(t)}\psi\|h\|\nabla^{g_0(t)}h|$ yields 
\begin{align*}
    \dt \int_{N\times\mathbb{R}} &\psi^2|h|^2d\mu_{g_0(t)}  + \int_{N\times \mathbb{R}} \psi^2|\nabla^{g_0(t)}h|^2d\mu_{g_0(t)}\\
    &\leq c(n+1)\int_{N\times\mathbb{R}}\left( |\nabla^{g_0(t)}\psi|^2+\dt\psi^2+\psi^2|Rm(g_0(t))|\right)|h|^2d\mu_{g_0(t)}.
\end{align*}

Let $(\tilde{x},R)\in N\times\mathbb{R}\times\mathbb{R}_+$ and, without loss of generality, assume $\tilde{x}=(x,0).$ We now choose an explicit cut-off function as
\begin{align*}
    \psi_{\tilde{x},R}(\Tilde{y},t):= \xi\left(\frac{d_{g_N(t)}(x,y)}{R}\right)\xi\left(\tfrac{l}{R}\right),
\end{align*}
where $\Tilde{y}=(y,l),$ $\xi:\mathbb{R}_+ \rightarrow \mathbb{R}_+$ is a smooth function satisfying $\xi_{|_{[0,1]}}=1,$ $\xi_{|_{[2,\infty)}}=0$ and $\displaystyle{\sup_{\mathbb{R}_+}|\xi'|\leq c.}$ Therefore, $\psi_{\tilde{x},R}$ is a Lipschitz function such that  ${| \nabla^{g_0(t)}\psi_{\tilde{x},R}|\leq \frac{c}{R}  }$ and $\displaystyle{0\leq \dt \psi_{\tilde{x},R}\leq \frac{c}{R\sqrt{t}} }$ by the following distortion estimate: for a type III Ricci flow $(M,g(t))$ with non-negative Ricci curvature, we have
 \begin{align*}
     d_{g}(x,y)- c(A_0)\left(\sqrt{t}-\sqrt{s}\right)\leq d_{g_0(t)}(x,y) \leq d_{g_0(s)}(x,y),
 \end{align*}
for $s\leq t$ and for $x,y\in M.$ For details on this estimate, see \cite[Lemma A.3]{Der-Lamm}.

Now we integrate the estimate above in time from $0$ to $R^2$ to obtain

\begin{align*}
    \int_{-R}^R& \int_{B_{g_N(R^2)}({x},R)}|h|^2d\mu_{g_N(R^2)}dr+\int_{0}^{R^2}\int_{-R}^R\int_{B_{g_N(t)}({x},R)}|\nabla^{g_0(t)}h|^2d\mu_{g_N(t)}drdt \\
    &\leq  \int_{-2R}^{2R}\int_{B_{g_N(0)}({x},R)}|h(0)|^2d\mu_{g_N(0)}dr \\
     &+ c(n)\int_{-2R}^{2R}\int_{0}^{R^2} \int_{B_{g_N(t)}({x},2R)} \left( \tfrac{1}{R^2}+\tfrac{1}{R\sqrt{t}}+ |Rm(g_0(t))|\right) |h(t)|^2d\mu_{g_N(t)}drdt.
\end{align*}

Recalling that $AVR(g_N(t))=AVR(g_N)>0,$ we get $Vol_{g_0(t)}(B_{{g_0(t)}}(x,s)) \geq AVR(g_0)s^{n+1} $  for any $s>0$ by the Bishop-Gromov theorem. Furthermore, Bishop-Gromov once again yields $Vol_{g_0(t)}(B_{{g_0(t)}}(x,s))\leq c(n)s^{n+1},$ where we are using that $Ric(g_0(t))\geq 0.$ Therefore, applying the $L^\infty$ estimate $\displaystyle{\sup_{t\geq 0}\|h(t)\|_{L^{\infty}(N\times \mathbb{R},{g_0}(t))}\leq c \|h_0\|_{L^{\infty}(\expander)}}$ as in \cite{Der-Lamm}, we have:

\begin{align*}
    & R^2 \frac{1}{R^{(n+1)+2}}\| \nabla h\|^2_{L^2(C({x},0,R))}\\
    &\leq c\|h(0)\|_{L^\infty(N\times \mathbb{R},g_0)}^2\left( 1 + \tfrac{1}{R^{n+1}}\int_{0}^{R^2}\int_{-2R}^{2R}\int_{B_{g_N(t)({x},2R)}}|Rm({g}_0(t))|d\mu_{{g}_N(t)}drdt \right) \\
    & = c\|h(0)\|_{L^\infty(N\times \mathbb{R},g_0)}^2\left( 1 + \tfrac{1}{R}\int_{0}^{R^2}\int_{-2R}^{2R}dr\dashint_{B_{g_N(t)({x},2R)}}|Rm({g}_N(t))|d\mu_{{g}_N(t)}dt \right) \\
    & \leq c\|h(0)\|_{L^\infty(N\times \mathbb{R},g_0)}^2\left(1+\tfrac{1}{R}\int_{-2R}^{2R}cdr\right)=c\|h(0)\|_{L^\infty(N\times \mathbb{R},g_0)}^2,
\end{align*}
where $c=c(n)$ and we are using that $|Rm(g_0(t))|_{g_0(t)}=|Rm({g}_N(t))|_{g_N(t)}.$ Furthermore, on the last line, the estimate on the average integral follows as in the proof of Lemma \ref{lemma31DerLamm}.
\end{proof}

\subsection{Heat kernel estimates}
The next step is to prove heat kernel estimates. Again, many of the estimates available for the heat kernel acting on functions over $(N,g_N,f_N)$ only use that the solution to the Ricci flow is type III, non-collapsed and has non-negative Ricci/scalar curvature. Thus, they still hold and will be useful for our set-up. We move on to proving on-diagonal bounds of the Lichnerowicz heat kernel. 

We recall the setup from \cite{Der-Lamm}. Let $(M,g(t))_{t\in [0,T)}$ be a complete Ricci flow with bounded curvature at all times. The heat equation coupled with Ricci flow is given by
\begin{align}\label{heatflow}
    & \dt u= \Delta_{g(t)}u + R(g(t)) u,\\
  \nonumber  & \dt g= -2 Ric(g(t)),
\end{align}
on $M\times (0,T).$ We also consider the conjugate heat equation:
\begin{align}\label{conjheatflow}
    &\frac{\partial}{\partial \tau}u=\Delta_{g(\tau)}u,\\
    \nonumber & \frac{\partial}{\partial \tau}u=2Ric(g(\tau))
\end{align}
on $M\times (0,t]$ where $\tau(s):= t-s$ for $t>0$ fixed. 

Following standard notation, we denote the heat kernel associated to \eqref{heatflow} by $K(x,t,y,s),$ for $0\leq s < t < T,$ and $x,y\in M.$ It is defined by:
\begin{align}\label{scalarkernel}
\nonumber    &\dt K(\cdot,\cdot,y,s)=\Delta_{g(t)}K(\cdot,\cdot,y,s)+ R(g(t))K(\cdot,\cdot,y,s),\\
    & \dt g=-2Ric(g(t)),\\
\nonumber    & \lim_{t\to s}K(\cdot,t,y,s)=\delta_y,
\end{align}
where $(y,s)\in M\times (0,T)$ are fixed. Now, if $(x,t)\in M \times(0,T)$ are fixed, then $K(x,t,\cdot,\cdot)$ is the heat kernel associated to the conjugate heat equation:
\begin{align}
    \nonumber & \frac{\partial}{\partial s}K(x,t,\cdot,\cdot)=\Delta_{g(\tau)}K(x,t,\cdot,\cdot),\\
    & \frac{\partial}{\partial s}g=2Ric(g(\tau)),\\
    \nonumber & \lim_{\tau \to 0}K(x,t,\cdot,\tau)=\delta_x.
\end{align}

We now state some estimates for the heat kernel $K_L$ associated to the Lichnerowicz operator. Recall that $K_L\in Hom(S^2T^*(N\times\mathbb{R}), S^2T^*(N\times\mathbb{R}))$ is $C^1$ with respect to time and $C^2$ with respect to space variables, and it satisfies, by definition,
\begin{align}\label{tensorkernel}
    \nonumber &\dt K_L(\cdot,\cdot,y,s)=L_t K_L(\cdot,\cdot,y,s),\\
    &\dt g_0=-2Ric(g_0(t),\\
    \nonumber & \lim_{t\to s}K_L(\cdot,t,y,s)=\delta_y,
\end{align}
for $y\in N\times\mathbb{R}$ and $s\in[0,\infty)$ fixed. On the other hand, fixing $(x,t),$ $K_L$ satisfies the conjugate backward heat equation:
\begin{align}
    \nonumber & \frac{\partial}{\partial s}K_L(x,t,\cdot,\cdot)=-L_sK_L(x,t,\cdot,\cdot)+R(g_0(s))K_L(x,t,\cdot,\cdot),\\
    & \frac{\partial}{\partial s}g=2Ric(g_0(\tau)),\\
    \nonumber & \lim_{\tau \to 0}K_L(x,t,\cdot,\tau)=\delta_x.
\end{align}

We first recall the following diagonal bound for the Lichnerowicz heat kernel on an expanding Ricci soliton with positive curvature operator from \cite[Theorem 4.5]{Der-Lamm}:
\begin{align}\label{diagbound}
    \| K_L(x,t,y,s)\|_{Hom(S^2T_y^*(N),S^2T_x^*(N))} \leq \frac{c(n+1,V_0,A_0)}{(t-s)^{\frac{n+1}{2}}},
\end{align}
for $0\leq s < t$ and $x,y \in N.$ Since $K_L$ splits as $K^N_L K^{\mathbb{R}}_L,$ where $K^N_L$ and $K^{\mathbb{R}}_L$ are the respective Lichnerowicz heat kernels of $(N,g_N(t))$ and $(\mathbb{R},g_{\mathbb{R}}(t))$ (see, for instance, \cite{grigoryan}), the bound above can be trivially extended to $\expander.$ Furthermore, we also obtain the following off-diagonal bound as a straightforward application of Deruelle--Lamm's result.

\begin{thm}\label{heatKernelThm}
Let $(N^n\times\mathbb{R},g_0(t))_{t\geq 0}$ be an expanding gradient Ricci soliton such that $(N,g_N(t))$ is a Ricci soliton with positive curvature operator. Then the heat kernel associated to \eqref{tensorkernel} satisfies the following Gaussian estimate:
\begin{align*}
    \| K_L(x,t,y,s)\| \leq \frac{c(n,V_0,A_0,D)}{(t-s)^{\frac{n+1}{2}}}\exp\left\{ -\frac{d^2_{g_0(s)}(x,y)}{D(t-s)}\right\},
\end{align*}
for $0\leq s<t,$ $x,y\in N\times\mathbb{R}$ and $D>4,$ where $A_0:= \sup_{N\times\mathbb{R}}|Rm(g_0)|_{g_0},$ $V_0=\text{AVR}(g_0)$ and the norm is $\|\cdot\|= \| \cdot\|_{Hom(S^2T_y^*(N\times\mathbb{R}),S^2T_x^*(N\times\mathbb{R}))}.$

\end{thm}
\begin{proof}

Let $\|\cdot\|_{g_0}$ be the norm on $Hom(S^2T^*(N\times\mathbb{R}),S^2T^*(N\times\mathbb{R})).$ Consider points $(x,r_x)$ and $(y,r_y)$ in $N\times\mathbb{R}.$  From Theorem 5.14 in \cite{Der-Lamm} we have
\begin{align*}
   \| K_L^N(x,t,y,s)\|_{g_N}\leq  \frac{c(n,V_0,A_0,D)}{(t-s)^{\frac{n}{2}}}\exp\left\{ -\frac{d^2_{g_N(s)}(x,y)}{D(t-s)}\right\}.
\end{align*}
Therefore,
\begin{align*}
    \| &K_L((x,r_x),t,(y,r_y),s)\|_{g_0}\leq \| K_L^N(x,t,y,s)\|_{g_N}\|K^{\mathbb{R}}_{g_{\mathbb{R}}}(r_x,t,r_y,s) \|_{g_{\mathbb{R}}}\\
    & \leq \frac{c(n,V_0,A_0,D)}{(t-s)^{\frac{n}{2}}}\exp\left\{ -\frac{d^2_{g_N(s)}(x,y)}{D(t-s)}\right\} \frac{1}{\sqrt{4\pi (t-s)}}\exp \left\{ -\frac{|r_x-r_y|^2}{4(t-s)}\right\}
\end{align*}
since $K^{\mathbb{R}}_{g_{\mathbb{R}}}(r_x,t,r_y,s)=\frac{1}{\sqrt{4\pi (t-s)}}\exp \left\{ -\frac{|r_x-r_y|^2}{4(t-s)}\right\}.$ Finally,
\begin{align*}
    \| K_L((x,r_x),t,(y,r_y),s)\|_{g_0} \leq \frac{c(n+1,V_0,A_0,D)}{(t-s)^{\frac{n+1}{2}}}\exp\left\{ -\frac{d^2_{g_0(s)}(x,y)}{D(t-s)},         \right\}.
\end{align*}

\end{proof}

\subsection{The inhomogeneous case}

Having obtained bounds for the homogeneous linear equation and good estimates for the heat kernel on $(N\times\mathbb{R},g_0(t))_{t\geq 0},$ we can proceed to estimate the inhomogeneous equation. We state below Theorem \ref{inhBound} and Theorem \ref{stabilityTheorem}, and refer the reader to \cite{Der-Lamm} for the details. It is important to note, however, the dependence on the decay of the heat kernel on $\mathbb{R},$ in the form of Theorem \ref{heatKernelThm}.
\begin{thm}\label{inhBound}
Let $R$ be in $Y.$ Then, any solution to $(\partial_t-L_t)h=R$ with $h(0)=h_0\in L^{\infty}(N\times\mathbb{R},g_0)$ of the form $h=K_L*R,$ i.e.,
\begin{align*}
    h(x,t)=\int_{0}^t\int_{N\times\mathbb{R}} <K_L(x,t,y,s),R(y,s)>d\mu_{g_0(s)}(y)ds,
\end{align*}

is in $X$ and satisfies 
\begin{align*}
    \|h\|_X \leq c(\|h_0\|_{L^{\infty}(N\times\mathbb{R},g_0)}+\|R\|_Y),
\end{align*}
for some positive uniform constant $c.$
\end{thm}
As a corollary of Theorem \ref{inhBound} and Lemma \ref{lemma31DerLamm}, we obtain the following.

\begin{thm}\label{stabilityTheorem}
Let $(N^n\times\mathbb{R},g_0(t))_{t\geq 0},$ where $g_0(t)=g_N(t)\otimes g_{\mathbb{R}}(t)$ and $n\geq 3,$ be an expanding gradient Ricci soliton such that $(N,g_N(t))$ is a Ricci soliton with positive curvature operator and quadratic curvature decay at infinity, i.e.,
\begin{align*}
    Rm(g_N)>0, \hspace{0.3cm} |Rm(g_N)|(x)\leq \frac{C}{1+d^2_{g_N}(p,x)}, \hspace{0.1cm} \forall x \in N,
\end{align*}
for some point $p\in N$ and a positive constant $C$ depending on $p.$ Then there exists $\varepsilon>0$ such that for any metric $g\in L^{\infty}(N\times\mathbb{R},g_0) $ satisfying $\|g-g_0\|_{L^{\infty}(N\times\mathbb{R},g_0)}\leq \varepsilon,$ there exists a global solution $(N\times\mathbb{R},g(t))_{t\geq 0}$ to the DeTurck Ricci flow with initial condition $g.$ Moreover,
\begin{enumerate}
    \item the solution is analytic in space and time and, for any $\alpha,\beta\geq 0$ the following holds:
    \begin{align*}
        \sup_{x\in N\times\mathbb{R},t>0}|(\sqrt{t}\nabla^{g_0(t)})^{\alpha}(t\partial_t)^{\beta}(g(t)-g_0(t))|_{g_0(t)}\leq c_{\alpha,\beta}|g(0)-g_0|_{L^{\infty}(N\times\mathbb{R},g_0)},
    \end{align*}
    \item the solution above is unique in $B_X(g_0,\varepsilon).$ 
\end{enumerate}
\end{thm}

\begin{rem}
We observe that the result above can be directly extended to $N\times\mathbb{R}^k,$ for any $k \in \mathbb{Z}_{+},$ since the heat kernel in $\mathbb{R}^k$ has the same decay estimates.
\end{rem}

\section{Uniform estimates for the Ricci flow starting from the model class}\label{uniformEstimates}

We now consider a complete Ricci flow $(M^n,g(t))_{t\in [0,T]},$ with $n\geq 4$ and $(M,g(0))\in \mathcal{M}(\delta, s,\Lambda)$ for some $\Lambda\geq \Lambda_0.$ The constant $\Lambda$ will separate our manifold into a \textit{conical} region and an \textit{expanding} region. In the conical region, the metric is $\delta-$close to the product metric on $C(X)\times\mathbb{R}$, and in the expanding region it is $\delta-$close to the metric on the expander $\expander$ at scale $s.$

\subsection{Conical region estimates}
In this section, we argue that if the initial data is in $\mathcal{M}(\delta,\Lambda,s),$ pseudolocality controls the flow in the conical region. To do so, we work with the Ricci-DeTurck flow. Let $\hat{g}(t)$ be a solution for the Ricci-DeTurck flow 
\begin{align}\label{RDTF-changingbackground}
\dt \hat{g}=-2Ric(\hat{g})+\mathcal{L}_{\mathcal{W}(\hat{g},\tilde{g})}\hat{g},
\end{align}
where $\displaystyle{\mathcal{W}(\hat{g},\tilde{g})_k=\hat{g}_{kl}\hat{g}^{ij}(\hat{\Gamma}_{ij}^{l}-\tilde{\Gamma}_{ij}^{l})},$ and $\tilde{g}(t)$ is chosen as follows. For $(M,\hat{g}(0))\in \mathcal{M}(\delta,\Lambda,s),$ we define
\begin{align}\label{background_metrics}
    \tilde{g}(t)=\xi(r_s)(\Phi_{s}^{-1})^* g_{\exp}(t+s) + (1-\xi(r_s))\hat{g}(0),
\end{align}
where $\xi: [0,1]\mapsto [0,1]$ is a fixed smooth, non-increasing function which is identically 1 in $[0,\frac{1}{2}]$ and $\xi=0$ in $[\frac{5}{8},1].$ Given a Ricci flow $(M,g(t))_{t\in [0,T]}$ with $\mathcal{M}(\delta,\Lambda,s),$ consider $\displaystyle{\psi:\left\{r_s\leq \tfrac{3}{4}\right\}\times [0,T] \to \left\{r_s\leq \tfrac{3}{4}\right\}}$ solving the harmonic map heat flow 
\begin{align}\label{HMRF}
    & \dt \psi=\Delta_{g(t),\tilde{g}(t)}\psi,\\
    & \psi|_{t=0}=id_{\{r_s\leq \frac{3}{4}\}},\\
    & \psi_{\{r_s= \frac{3}{4}\}\times [0,T]}=id_{\{r_s= \frac{3}{4}\}}.
\end{align}

We assume that $T\leq 1$ and is small enough so that $g(t)$ and $\psi_t$ are smooth for every $t\in [0,T]$ and $\psi_t$ is a diffeomorphism. Then it is well known that $\displaystyle{\hat{g}(t)=(\psi_t^{-1})^*g(t)}$ solves \eqref{RDTF-changingbackground}. Assuming an \textit{a priori} bound on the gradient of $\psi,$ the lemma below (see \cite[Lemma 3.1]{conicalsing}) implies that $\hat{g}(t)$ is controlled until a definite positive time. The proof follows from Perelman's pseudolocality, using the fact that the metric is close to $G=g_c\otimes \gcan$ in the region defined below. We refer the reader to \cite{conicalsing} for the details.

\begin{lem}\label{conelemma}
Fix $B,\alpha>0.$ There exist $\delta_2(\alpha)>0,$ $\gamma_2(B,\alpha)>1$ and a constant $C(G)>0$ such that the following holds. Let $(M,g(t))_{t\in [0,T]}$ be a complete Ricci flow with bounded curvature and initial data satisfying $(M,g(0))\in \mathcal{M}(\delta_2,\Lambda,s),$ for some $\Lambda\geq \Lambda_0$ and $s\leq \frac{1}{32(\Lambda+1)^2}.$ Let $\tilde{g},$ $\psi$ and $\hat{g}(t)=(\phi_t^{-1})^*g(t)$ be as above and define the following conical region:

\begin{align}
    \mathcal{D}_{\gamma,\Lambda,s}^{cone}=\left\{ (p,t)\in \{r_s\leq \tfrac{3}{4}\}\times [0,(32\gamma)^{-1}], r_s(p)\geq \sqrt{\gamma t + s\Lambda^2}\right\}
\end{align}
for some $\gamma \geq \gamma_2$ and suppose $|\nabla\psi|_{g,\tilde{g}}\leq B$ in $\{ r_s \leq \tfrac{3}{4}\}\times [0,\min\{(32\gamma)^{-1},T\}].$ Then the following estimates are valid in $\mathcal{D}_{\gamma,\Lambda+1,s}^{cone} \cap (M\times [0,T]):$
\begin{align}\label{ConStab}
    |\hat{g}-\tilde{g}|_{\tilde{g}} +r_s|\tilde{\nabla}\hat{g}|_{\tilde{g}}<\alpha,
\end{align}
\begin{align}
    \sum_{j=0}^{2}r_s^{2+j}\left|(\nabla^g)^j Rm(g)\right|_g  \leq C.
\end{align}
\end{lem}

\subsection{Local stability for the expanding soliton}

We now localise our stability result from section \ref{stabilitySection}. For $(\expander,g_0(t))_{t\geq 0},$ with $g_0(0)=g_0=g_N\otimes\gcan,$ we define the following regions. First, the interior region
\begin{align*}
    D_{\lambda_0} = \left\{ (x,l,t)\in \expander\times [0,T];\hspace{0.1cm} \mathbf{r}(x)\leq 2\sqrt{\gamma t +(\Lambda+1)^2},\hspace{0.1cm} l \in [-\lambda_0/2, \lambda_0/2]\right\}.
\end{align*}
We will also need the annular region
\begin{align*}
    A_{\lambda_0}= \big\{ (x,l,t)\in \expander\times[0,T]; \hspace{0.1cm} &\mathbf{r}(x)\in [\sqrt{\gamma t +(\Lambda+1)^2},2\sqrt{\gamma t +(\Lambda+1)^2}],\\
    & l\in [-\lambda_0/2, \lambda_0/2]\big\},
\end{align*}
and the boundary region
\begin{align*}
  B_{\lambda_0}=  \big\{ (x,l,t)\in \expander\times [0,T];\hspace{0.1cm} & \mathbf{r}(x)\leq 2\sqrt{\gamma t +(\Lambda+1)^2},\\
  &l\in [-\lambda_0/2,-\lambda_0/2 +1]\cup [\lambda_0/2 -1,\lambda_0/2]  \big\}. 
\end{align*}
Before our main lemma, we state the following technical lemma, which is just an adaptation of \cite[Lemma 4.3]{conicalsing} to our setting. 

\begin{lem}\label{lemma43}
Let $(N,g_N,f)$ be an asymptotically conical gradient Ricci expander with positive curvature operator and let $(g_0(t))_{t\geq 0}$ be the induced Ricci flow on $\expander$ with $g_0(0)=g_0.$ There is a $C(g_N)>0$ such that for every $\Lambda\geq \Lambda_0$ the following holds. Define $A_{\lambda_0}$ as above, for some $\gamma \geq 1.$ Then, if the tensors $h_1,$ $h_2$ are supported in $A_{\lambda_0}$ and satisfy 
\begin{align*}
    |h_1|_{g_0(t)}+|h_2|_{g_0(t)}^2 \leq \frac{D}{t+\gamma^{-1}(\Lambda+1)^2},
\end{align*}
we have
\begin{align*}
    \|h_1 + \nabla^{g_0(t)}h_2\|_{Y_T}\leq C(g_N)D.
\end{align*}

\end{lem}
The proof of the next lemma will follow \cite[Lemma 4.2]{conicalsing} and we refer the reader to their paper for the details of some steps. We will focus on the steps where our proof deviates from theirs.

\begin{lem}\label{localstab}
Let $\displaystyle{\left(\expander,g_0, f_0\right)}$ be a gradient Ricci expander, where $(N,g_N,f_N)$ is an asymptotically conical gradient Ricci expander with positive curvature operator, and let $g_0(t),$ $t\geq 0,$ be the induced Ricci flow with $g_0(0)=g_N\otimes\gcan.$ There exists $\delta_1(g_N)>0$ with the following property. Let $\Lambda\geq \Lambda_0,$ $\gamma \geq 1,$ $0<\delta_2<\delta_1$  and $\mathbf{r}(x)=2\sqrt{f_N(x)}$ for all $x\in N.$ Let $g(t),$ $t\in[0,T],$ be a solution to the Ricci-DeTurck flow
\begin{align*}
    \dt g(t)=-2Ric(g(t)) + \mathcal{L}_{W(g(t),g_0(t))}g(t)
\end{align*}
on $D_{\lambda_0},$ with $T\leq \min\{\tfrac{1}{32\gamma},\tfrac{1}{\lambda_0}\},$ and assume that 
\begin{align*} 
   0< H := \max \bigg\{ \sup_{D_{\lambda_0}\cap \{t=0\}}|g-g_0|_{g_0}, &\sup_{A_{\lambda_0}} ( |g-g_0|_{g_0} + \mathbf{r}|\nabla^{g_0}g|_{g_0})\bigg\} \leq \delta_1.
\end{align*}
We also assume that for some constant $0<M\leq 1,$
\begin{align*}
    \sup_{B_{\lambda_0}}(|g-g_0|_{g_0} + \sqrt{1+t}|\nabla^{g_0}g|_{g_0}) \leq M.
\end{align*}
Then, if $\lambda_0\geq \frac{M}{\delta_2}$ and $0<H\leq \delta_2\leq \delta_1,$ defining 

\begin{align*}
    {D}^{'}_{\lambda_0}=D_{\lambda_0}\cap \left( \left\{ \mathbf{r}(x)\leq \tfrac{3}{2}\sqrt{\gamma t +(\Lambda+1)^2}, -\lambda_0/4 \leq l \leq \lambda_0/4 \right\}\times [0,T] \right),
\end{align*}
we have
\begin{align*}
    \sup_{D^{'}_{\lambda_0}}\left| (\sqrt{t}\nabla^{g_0})^a(t\partial_t)^b(g-g_0)\right|_{g_0}\leq C_{a,b}(g_N),
\end{align*}
for any non-negative indices $a,b.$ Furthermore, for all $k = 0,1,2,...,$ there exist $C_k(g_N)$ and $0< \delta_k(g_N)\leq \delta_1$ such that if $H\leq \delta_k,$ then
\begin{align*}
    \sup_{D_{\lambda_0}'}\left| (\sqrt{t}\nabla^{g_0})^a(t\partial_t)^b(g-g_0)\right|_{g_0}\leq C_k H,
\end{align*}
provided $a+2b \leq 2k.$
\end{lem}
\begin{proof}
We start with an arbitrary $\lambda_0 > 2,$ to be chosen later in the proof. Let $\xi_1 : \mathbb{R} \to \mathbb{R}$ and $\xi_2: [0,\infty)\to \mathbb{R},$ $0\leq \xi_i \leq 1,$ for $i=1,2,$ be two smooth cut-off functions such that 
\begin{align*}
    & {\xi_1}{|_{\left[-\tfrac{\lambda_0}{2}+1,\tfrac{\lambda_0}{2}-1\right]}}\equiv 1, \hspace{0.5cm} {\xi_1}{|_{[\lambda_0/2,\infty)}}\equiv {\xi_1}{|_{(-\infty, -\lambda_0/2]}}\equiv 0,\\
    & {\xi_2}{|_{[0,1]}}\equiv 1, \hspace{0.5cm} {\xi_2}{|_{[2,\infty)}}\equiv 0, \\
    & |{\xi_i}'|+|{\xi_i}''| \leq C_{{\xi}}, 
\end{align*}
for some universal constant $C_{\xi}>0.$ Define the following cut-off function in $D_{\lambda_0}:$
\begin{align*}
    \chi(x,l,t)=\xi_1\left(l\right)\xi_2 \left(\frac{\mathbf{r}(x)}{\sqrt{\gamma t +(\Lambda+1)^2}}\right).
\end{align*}
 Then,
\begin{align*}
    \dt \chi = -\xi_1\left(l\right)\xi_2^{'} \left(\frac{\mathbf{r}(x)}{\sqrt{\gamma t +(\Lambda+1)^2}}\right)\frac{\mathbf{r}(x)\gamma}{2(\gamma t +(\Lambda+1)^2)^{\tfrac{3}{2}}},
\end{align*}
and so, $\displaystyle{|\dt \chi| \leq C_{\xi}\frac{1}{\gamma t + (\Lambda+1)^2}}$ on $A_{\lambda_0}$ and, outside $A_{\lambda_0},$ $\dt \chi=0.$ In particular, $\dt \chi=0$ on $B_{\lambda_0}\backslash A_{\lambda_0}.$ Similarly, 
\begin{align*}
    \nabla^{g_0(t)}\chi &= \nabla^{g_0(t)}\left( \xi_1\left(l\right)\xi_2 \left(\frac{\mathbf{r}(x)}{\sqrt{\gamma t +(\Lambda+1)^2}}\right)\right) \\
    &=\xi_1^{'}\left(l\right)\xi_2 \left(\frac{\mathbf{r}(x)}{\sqrt{\gamma t +(\Lambda+1)^2}}\right)e_{\mathbb{R}} \\
    &+ \xi_1\left(l\right)\xi_2^{'}  \left(\frac{\mathbf{r}(x)}{\sqrt{\gamma t +(\Lambda+1)^2}}\right)\frac{\nabla^{g_N(t)}\mathbf{r}(x)}{\sqrt{\gamma t +(\Lambda+1)^2}},
\end{align*}
with $\xi_2^{'}=0$ outside of $A_{\lambda_0}$ (in particular, in $B_{\lambda_0}\backslash A_{\lambda_0}).$ Lemma \ref{lemma21anal} implies that $\displaystyle{|\nabla^{g_0(t)}\mathbf{r}|^2\leq 2}$ in $A_{\lambda_0}.$ Thus,
\begin{align*}
    |\nabla^{g_0(t)}\chi|_{g_0(t)}^2 \leq C(n,\xi) \frac{1}{t +\gamma^{-1}(\Lambda+1)^2} +C(n,\xi)
\end{align*}
in $A_{\lambda_0},$ and 
\begin{align*}
    |\nabla^{g_0(t)}\chi|_{g_0(t)}^2 \leq C(n,\xi)
\end{align*}
in $B_{\lambda_0}\backslash A_{\lambda_0},$ with $supp(\nabla^{g_0(t)}\chi)=A_{\lambda_0}\cup B_{\lambda_0}.$

Finally, again using Lemma \ref{lemma21anal} to get the bound $\displaystyle{\mathbf{r}\Delta_{g_0}\mathbf{r}\leq 4(n-1)},$ we have 
\begin{align*}
    \Delta_{g_0(t)}\chi&= \xi_1\left(l\right)\bigg[\xi_2^{''}\left(\frac{\mathbf{r}(x)}{\sqrt{\gamma t +(\Lambda+1)^2}}\right)\frac{1}{\gamma t +(\Lambda+1)^2}|\nabla^{g_N(t)}\mathbf{r}|^2\\
   & + \xi_2^{'}\left(\frac{\mathbf{r}(x)}{\sqrt{\gamma t +(\Lambda+1)^2}}\right)\frac{1}{\sqrt{\gamma t +(\Lambda+1)^2}}\Delta_{g_N(t)}\mathbf{r}\bigg] \\
    &+ \xi_2\left(\frac{\mathbf{r}(x)}{\sqrt{\gamma t +(\Lambda+1)^2}}\right)\xi_1^{''}\left(l\right),
\end{align*}
hence, in $A_{\lambda_0}$ we have
\begin{align*}
    |\Delta_{g_0(t)}\chi|\leq C(n,\xi)\frac{1}{t +\gamma^{-1}(\Lambda+1)^2} +C(n,\xi)
\end{align*}
since $\xi_2^{'}=0$ in the set $\{ \mathbf{r}\leq \sqrt{\gamma t + (\Lambda+1)^2}\}.$ Finally,
\begin{align*}
    |\Delta_{g_0(t)}\chi|\leq C(n,\xi)
\end{align*}
in $B_{\lambda_0}\backslash A_{\lambda_0},$ with $supp(\Delta_{g_0(t)}\chi)=A_{\lambda_0}\cup B_{\lambda_0}.$ Using that $\mathbf{r}(x)\geq \sqrt{\gamma t +(\Lambda+1)^2}$ in $A_{\lambda_0},$ we can apply Lemma \ref{lemma21anal} to get
\begin{align}\label{R1bound}
    |\nabla^{g_0(t)}g(t)|_{g_0(t)}\leq \frac{H}{\sqrt{t + \gamma^{-1}(\Lambda+1)^2}}
\end{align}
in $A_{\lambda_0}.$

We denote $h(t)= g(t)-g_0(t)$ to write  the evolution equation of $h$ as
\begin{align*}
    (\partial_t - L_t)h= R_0[h] + \nabla R_1[h],
\end{align*}
where, as in Section \ref{stabilitySection},
\begin{align*}
    & L_t h_{ij}= \Delta_{g_0(t)}h_{ij}+ 2Rm(g_0(t))_{ijkl}h_{kl}-Ric(g_0(t))_{ik}h_{kj}-Ric(g_0(t))_{jk}h_{ki},\\
    & R_0[h]= Rm(g_0(t))*h*h + O(h^3)*Rm(g_0(t))+ g^{-1}*g^{-1}*\nabla^{g_0(t)}h * \nabla^{g_0(t)}h,\\
    & \nabla R_1[h]= \nabla_p^{g_0(t)}\left( \left( ( g_0(t)+h(t))^{pq}- (g_{0}(t))^{pq}\right) \nabla_q^{g_0(t)}h\right),
\end{align*}
with $R_1[h]=( ( g_0(t)+h(t))^{pq}- (g_{0}(t))^{pq})\nabla_q^{g_0(t)}h$ and $|O(h^3)|_{g_0(t)}\leq C |h(t)|^3_{g_0(t)}.$ 

Multiplying $h$ by $\chi^2,$ we can deduce an evolution equation for the product:
\begin{align}\label{hEvol}
    (\partial_t-L_t)(\chi^2h) &= \left( 2\chi\partial_t\chi - 2\chi\Delta_{g_0(t)}\chi - 2|\nabla^{g_0(t)}\chi|^2 \right) h + \chi^2R_0[h]\\
   \nonumber & + \nabla^{g_0(t)}(\chi^2R_1[h]) - 2\chi \nabla^{g_0(t)}\chi * \nabla^{g_0(t)}h - 2\chi\nabla^{g_0(t)}\chi*R_1[h].
\end{align}
Working with the same spaces and norms as Gianniotis--Schulze (see \cite[page 22]{conicalsing}), we define the following two terms:
\begin{align*}
    S_1[h]=\chi^2R_0[h]+ \nabla^{g_0(t)}(\chi^2R_1[h]),
\end{align*}
and
\begin{align*}
S_2[h]=\left( 2\chi\partial_t\chi - 2\chi\Delta_{g_0(t)}\chi - 2|\nabla^{g_0(t)}\chi|^2 \right)h & - 2\chi \nabla^{g_0(t)}\chi * \nabla^{g_0(t)}h \\
& - 2\chi\nabla^{g_0(t)}\chi*R_1[h].
\end{align*}
We will estimate $S_1$ and $S_2$ separately. It will also be helpful to split our analysis of the terms involving the cut-off function into $A_{\lambda_0}$ and $B_{\lambda_0}.$

We first estimate $S_2[h],$ noting that it is supported on $A_{\lambda_0}\cup B_{\lambda_0}.$ In $A_{\lambda_0},$ we can apply \eqref{R1bound} and the estimates above for the cut-off function to get
\begin{align*}
    |S_2[h]|_{g_0(t)}&\leq\frac{C H }{t+\gamma^{-1}(\Lambda+1)^2},
\end{align*}
since $|h|_{g_0(t)}\leq H$ in $A_{\lambda_0}.$ This is enough to apply Lemma \ref{lemma43} and bound the norm of $S_2$ restricted to $A_{\lambda_0}$ by $C\delta_2.$ Working on $B_{\lambda_0},$ we have
\begin{align*}
    |S_2[h]|_{g_0(t)} \leq & \left(2\chi|\dt\chi|+2\chi |\Delta_{g_0(t)}\chi|+2|\nabla^{g_0(t)}\chi|^2   \right)|h|_{g_0(t)} +2\chi|\nabla^{g_0(t)}\chi||\nabla^{g_0(t)}h| \\
    & + 2\chi|\nabla^{g_0(t)}\chi||R_1[h]|\\
    & \leq C|h| +C|\nabla^{g_0(t)} h| \leq C \left( M +\frac{M}{\sqrt{1+t}}\right).
\end{align*}

We turn our attention to $S_1[h]$ for a moment, starting with $\chi^2R_0[h].$ First,
\begin{align*}
    | & \chi^2(h*h +O(h^3)Rm(g_0(t))) |_{g_0(t)}  \leq C \chi^2 |h|^2_{g_0(t)}|Rm(g_0(t))|_{g_0(t)} \\
    & \leq C|\chi^2 h|^2_{g_0(t)}|Rm(g_0(t))|_{g_0(t)} + C\chi^2(1-\chi^2)|h|^2_{g_0(t)}|Rm(g_0(t))|_{g_0(t)}.
\end{align*}
The first term on the RHS will be controlled by Theorem \ref{inhBound}. For the second term, observe that $\chi^2(1-\chi^2)$ is only non-zero in $A_{\lambda_0}\cup B_{\lambda_0}.$ Working on $A_{\lambda_0}$ first, we have
\begin{align*}
     C\chi^2(1-\chi^2)|h|^2_{g_0(t)}|Rm(g_0(t))|_{g_0(t)}\leq  \frac{C(g_0)\chi^2(1-\chi^2)H}{t+\gamma^{-1}(\Lambda+1)^2},
\end{align*}
since $|h|_{g_0(t)}\leq H$ in $A_{\lambda_0}$ and
\begin{align*}
    |Rm(g_0(t))|_{g_0(t)}=|Rm(g_N(t))|_{g_N(t)}\leq \frac{C(g_N)}{\mathbf{r}^2} \leq \frac{C(g_N)}{\gamma t + (\Lambda+1)^2}
\end{align*}
in $A_{\lambda_0}$ due to Lemma \ref{lemma21anal}. On $B_{\lambda_0},$ we observe that the norm of the curvature tensor is still bounded, now by $\displaystyle{\tfrac{C}{1+t}}$ since $(\expander, g_0(t))$ is a type III solution. Therefore, 
\begin{align}\label{badterm1}
    C\chi^2(1-\chi^2)|h|^2_{g_0(t)}|Rm(g_0(t))|_{g_0(t)}\leq \frac{C(g_0)\chi^2(1-\chi^2)M}{1+t},
\end{align}
since $|h|_{g_0(t)}\leq M$ in $B_{\lambda_0}.$

For $\chi^2g^{-1}*g^{-1}*\nabla^{g_0(t)}h * \nabla^{g_0(t)}h$ we have
\begin{align*}
    \chi^2g^{-1}*g^{-1}*\nabla^{g_0(t)}h*\nabla^{g_0(t)}h&=  \chi^2(1-\chi^2)g^{-1}*g^{-1}*\nabla^{g_0(t)}h * \nabla^{g_0(t)}h\\
    &+ g^{-1}*g^{-1}*\nabla^{g_0(t)}(\chi^2h)*\nabla^{g_0(t)}(\chi^2h)\\
    &+ g^{-1}*g^{-1}*\chi^2*\nabla^{g_0(t)}\chi*\nabla^{g_0(t)}\chi*h*h\\
    & + g^{-1}*g^{-1}*\chi^3*\nabla^{g_0(t)}\chi*\nabla^{g_0(t)}h*h.
\end{align*}
Thus,
\begin{align*}
    |\chi^2g^{-1}*g^{-1}* & \nabla^{g_0(t)}h*\nabla^{g_0(t)}h|_{g_0(t)} \leq C |\nabla^{g_0(t)}(\chi^2 h)|^2_{g_0(t)}\\
    & + \chi^2(1-\chi^2)|\nabla^{g_0(t)}h|^2_{g_0(t)} +C\chi^2|\nabla^{g_0(t)}\chi|^2_{g_0(t)}|h|^2_{g_0(t)}\\
    &+C\chi^3|\nabla^{g_0(t)}\chi|_{g_0(t)}|\nabla^{g_0(t)}h|_{g_0(t)}|h|_{g_0(t)}.
\end{align*}
Again, the first term is controlled by Theorem \ref{inhBound}. Working on $A_{\lambda_0},$ the three last terms are bounded by $\displaystyle{\frac{CH}{t+\gamma^{-1}(\Lambda+1)^2}}.$ On $B_{\lambda_0},$ we bound the same terms by
\begin{align}\label{badterm2}
\nonumber  \chi^2(1-\chi^2)|\nabla^{g_0(t)}h|^2_{g_0(t)} &+C\chi^2|\nabla^{g_0(t)}\chi|^2_{g_0(t)}|h|^2_{g_0(t)} \\
  &+C\chi^3|\nabla^{g_0(t)}\chi|_{g_0(t)}|\nabla^{g_0(t)}h|_{g_0(t)}|h|_{g_0(t)}\leq
\frac{CM^2}{1+t}+ CM^2 + \frac{CM^2}{\sqrt{1+t}}.
\end{align}

For the last term, $\displaystyle{\chi^2R_1[h]},$ we consider
\begin{align*}
    \chi^2R_1[h]=\chi^2\left[\left( (g_0(t)+h(t))^{pq}- (g_0(t))^{pq}\right)\nabla_q^{g_0(t)}h \right].
\end{align*} 
Thus,
\begin{align*}
    |\chi^2 R_1[h]|_{g_0(t)} & \leq  C|\chi^2h|_{g_0(t)}|\nabla^{g_0(t)}(\chi^2h)|_{g_0(t)}+C(1-\chi^2)\chi^2|h|_{g_0(t)}|\nabla^{g_0(t)}h|_{g_0(t)}\\
    & + C|h|^2_{g_0(t)}\chi^3|\nabla^{g_0(t)}\chi|_{g_0(t)}\\
    & \leq C|\chi^2h|_{g_0(t)}|\nabla^{g_0(t)}(\chi^2h)|_{g_0(t)}+C(1-\chi^2)\chi^2|h|_{g_0(t)}|\nabla^{g_0(t)}h|_{g_0(t)}\\
    & + C|h|^2_{g_0(t)}|\nabla^{g_0(t)}\chi|_{g_0(t)}.
\end{align*}
As above, the first term is controlled by Theorem \ref{inhBound}. The two last terms are supported on $A_{\lambda_0}\cup B_{\lambda_0},$ so we can directly bound them by $\frac{C H }{\sqrt{t+\gamma^{-1}(\Lambda+1)^2}}$ on $A_{\lambda_0},$ and work as before to bound them on $B_{\lambda_0}:$
\begin{align}\label{badterm3}
    C(1-\chi^2)\chi^2|h|_{g_0(t)}|\nabla^{g_0(t)}h|_{g_0(t)} +C|h|^2_{g_0(t)}|\nabla^{g_0(t)}\chi|_{g_0(t)} \leq \frac{CM^2}{\sqrt{1+t}} + CM^2.
\end{align}

Observe, then, that applying Theorem \ref{inhBound} and Lemma \ref{lemma43} allows us to bound the $Y$-norm of every term by $C\delta_2,$ except for the terms in (\ref{badterm1}-\ref{badterm2}-\ref{badterm3}) restricted to $B_{\lambda_0}.$ If we group these terms on a new tensor that we call $S_0,$ supported in $B_{\lambda_0},$ we get
\begin{align*}
    |S_0|_{g_0(t)}\leq CM +\frac{CM}{1+t} +\frac{CM}{\sqrt{1+t}},
\end{align*}
since $M\leq 1,$ where $C=C(n,g_0).$ We need weighted $L^1$ and $L^{\tfrac{n+4}{2}}$ bounds for $S_0.$ Since we only need to restrict ourselves to $B_{\lambda_0},$ we can directly estimate the weighted $L^1$ norm of $S_0$ on a parabolic ball $P((x,l),R)$ for $(x,l,t)\in B_{\lambda_0}.$ We have
\begin{align*}
   \frac{1}{R^{n+1}} \norm{S_0}_{L^1(P((x,l),R))}\leq & \frac{CM}{R^{n+1}}\int_{0}^{R^2}\int_{P((x,l),R)\cap(B_{\lambda_0}\cap \expander\times \{t\})}\left(1+ \frac{1}{1+t}+\frac{1}{\sqrt{1+t}}\right)d\mu_{g_0(t)}dt \\
   & \leq \frac{CM}{R^{n+1}}\int_{0}^{R^2}\int_{-\tfrac{\lambda_0}{2}}^{-\tfrac{\lambda_0}{2} +1}\left(1+ \frac{1}{1+t}+\frac{1}{\sqrt{1+t}}\right) C(g_N)R^n dldt\\
  & +\frac{CM}{R^{n+1}}\int_{0}^{R^2}\int_{\tfrac{\lambda_0}{2}-1}^{\tfrac{\lambda_0}{2}}\left(1+ \frac{1}{1+t}+\frac{1}{\sqrt{1+t}}\right) C(g_N)R^n dldt\\
  & \leq \frac{CM}{R}\int_{0}^{R^2}Cdt = CMR  \leq C\delta_2,
\end{align*}
where we used Bishop-Gromov to control the volume of the ball in $(N,g_N)$ from above, and our assumption on $\displaystyle{R \leq T' < \min\{T,\tfrac{1}{32\gamma},\tfrac{1}{\lambda_0}\}}$ and $\displaystyle{\tfrac{M}{\lambda_0}\leq \delta_2.}$   The $L^{\tfrac{n+4}{2}}$ can be estimated directly in the same way. 

Finally, applying Lemma \ref{lemma43}, Lemma \ref{lemma31DerLamm} and the estimate above yields
\begin{align*}
    \|S_1[h]\|_{Y_{T'}} +\|S_2[h]\|_{Y_{T'}} \leq C\left(\|\chi^2h\|^2_{X_{T'}} + \delta_2 \right).
\end{align*}
Together with Theorem \ref{inhBound}, this yields:
\begin{align*}
    \|\chi^2 h\|_{T'}\leq C(\|\chi^2 h\|_{T'}^{2}+ \delta_2).
\end{align*}
So for every $T'\leq T$ such that $\|\chi^2 h\|_{T'}\leq \frac{1}{2C},$ we have
\begin{align*}
    \|\chi^2 h\|_{T'}\leq C\delta_2.
\end{align*}
Therefore, since
\begin{align*}
    \lim_{T'\to 0}(\|\chi^2 h\|_{T'} - \sup_{\expander\times[0,T']}
|\chi^2 h|_{g_0})=0 \quad \text{and} \quad \lim_{T'\to 0}\sup_{\expander\times[0,T']}|h|_{g_0}\leq H \leq \delta_2,
\end{align*}
if $\max\{\delta_2,C\delta_2\}\leq \frac{1}{2C},$ we have $\displaystyle{ \|\chi^2 h\|_{T}\leq C\delta_2.}$

In order to obtain the decay estimates and finish our proof, we utilise a local argument and scaling. The idea is to first work on a small enough region, where the decay estimates will follow directly from parabolicity, making sure this small enough region is contained in $D_{\lambda_0}.$ We then use the rescaling nature of the Ricci flow to make sure that the estimates in fact work for every point in $D'_{\lambda_0}.$ The first step is proved in the following two claims.

\begin{claim}
There exist $0< r_0 < 1,$ $\varepsilon_0>0,$ and constants $C_{a,b}>0$ such that the following holds. Let $x_0\in N,$ $t_0\in (0,1],$ $0< r < \min(\sqrt{t_0},r_0),$ and consider $g(t)$ a solution to the Ricci-DeTurck flow with background metric $g_0(t)$ on 
\begin{align*}
    C((x_0,0),t_0,r):= \displaystyle{\bigcup_{t\in (t_0-r^2,t_0)}} \left( B_{g_N(t)}(x_0,r)\times(-r_0,r_0)\right)\times \{t\},
\end{align*}
with $|g(t)-g_0(t)|_{g_0(t)}\leq \varepsilon_0.$ Then
\begin{align*}
    |(r\nabla^{g_0})^{a} (r^2\partial_t)^b(g-g_0)|_{g_0}((x_0,0),t_0)\leq C_{a,b}.
\end{align*}
Furthermore, for every $k\in \mathbb{N}$ there exists $0<\varepsilon_k\leq \varepsilon_0$ such that if $|g(t)-g_0(t)|_{g_0(t)}\leq \varepsilon_k$ on $C(x_0,t_0,r),$ then there exists a constant $C_{a,b}^{'} >0$ such that
\begin{align*}
    |(r\nabla^{g_0})^a(r^2\partial_t)^b(g-g_0)|_{g_0}(x_0,t_0)\leq C_{a,b}^{'} \sup_{C((x_0,0),t_0,r)}|g(t)-g_0(t)|_{g_0(t)},
\end{align*}
as long as $a+2b\leq k.$
\end{claim}

For $r_0$ sufficiently small, we can find a coordinate system such that $g_0(t)$ is well controlled in $B_{g_N(0)}(x_0,r)\times(-r_0,r_0)$ for $t\in [0,1].$ The estimate then follows directly from local parabolic estimates. For details, see \cite[Proposition 2.5]{hypstab}.

\begin{claim}
There exists $0<\delta <1,$ independent of $\gamma$ and $\Lambda,$ such that for any $(x,l,t)\in D_{\lambda_0}',$ we have
\begin{align*}
    C((x,l),t,\sqrt{\delta t})\subset D_{\lambda_0}.
\end{align*}
\end{claim}
Since the metric is not changing in the $\mathbb{R}-$direction, we first choose $0<r_0 < \lambda_0/2.$ This guarantees the inclusion in the $\mathbb{R}-$direction. For the details of the inclusion for the ball in $N,$ see \cite[Lemma 4.2]{conicalsing}.

We now move on to the decay estimates for $D^{'}_{\lambda_0}.$ If $0<t<1,$ this follows directly from the two claims above. Therefore, fix $(x_0,l_0,t_0) \in D^{'}_{\lambda_0},$ with $1\leq t_0 \leq T.$ If we define $\eta:= \tfrac{2}{t_0}$ and the rescaling given by $\displaystyle{g^{\eta}(t):=\eta \phi_{\eta}^*g(\tfrac{t}{\eta})},$ then $g_0(t)=g^{\eta}_0(t)$ and $g^{\eta}$ solves the Ricci-DeTurck flow on 
\begin{align*}
    D^{\eta}_{\lambda_0}= \bigg\{ (x,l,t) \in \expander\times[0,1]; \hspace{0.1cm}  x\in \phi_{N}^{\eta^{-1}}\big(\{ &\mathbf{r}(x) \leq 2\sqrt{\tfrac{\gamma}{\eta}t + (\Lambda+1)^2}\}\big),\\
    &l \in [-\lambda_0/2,\lambda_0/2] \bigg\},
\end{align*}
where the $\mathbb{R}-$direction is not changing since rescaling by $\displaystyle{\frac{1}{\eta}}$ cancels out with $\phi_{\mathbb{R}}^{\eta^{-1}}.$ Therefore, the second claim holds for the new $\delta$ given by $\delta\eta.$ We can then apply the first claim to obtain the same estimates as before. Rescaling back to the original solution, we get the decay estimates from the statement.
\end{proof}

\subsection{Estimates on the expanding region}\label{expandingSubsection}

Finally, we prove estimates for the expanding region of our solution to the Ricci flow starting from $\mathcal{M}(\delta,\Lambda, s).$ Assuming the bounds on the conical region and an \textit{a priori} bound on the horizontal ends of a sub-region of the expanding region, we control the evolution of our solution in the interior of the sub-region. Overlapping such sub-regions, we obtain the needed estimates. From now on, we fix $p_0=\Phi_s(q_{\max},0)\in \Gamma_s$ and $\lambda_0 >0.$

\begin{lem}\label{expEstimate}
For every $\varepsilon >0$ and every $k\in \mathbb{N},$ there exists $\delta_0=\delta_0(g_N,\varepsilon,k)>0$ such that if $(M,g(t))_{t\in [0,T]}$ is a complete Ricci flow with bounded curvature, and initial data satisfying $(M,g(0))\in \mathcal{M}(\delta,\Lambda, s)$ for $\delta\leq \delta_0,$ $\Lambda>0$ and $s\leq \frac{1}{32(\Lambda+1)^2},$ then the following holds. Let $\hat{g}(t)=(\psi_t^{-1})^*g(t)$ be the Ricci-DeTurck flow in $\{ r_s\leq 3/4\}$ and let $\Tilde{g}(t)$ be as in \eqref{background_metrics}. Define the expanding region as:
\begin{align*}
    \mathcal{D}_{\gamma,\Lambda,s}^{exp}= \left\{ (x,t)\in M\times [0,\frac{1}{32\gamma}]\hspace{0.2cm} ; \hspace{0.2cm} r_s(x)\leq \tfrac{3}{2}\sqrt{\gamma t +s(\Lambda+1)^2}   \right\},
\end{align*}
and assume that estimate \eqref{ConStab} holds in $\mathcal{D}_{\gamma,\Lambda+1,s}^{cone} \cap (M\times [0,T])$ for some $\gamma\geq 1.$ Then we have
\begin{align}
    (t+s)^{j/2}|(\nabla^{\tilde{g}})^j(\hat{g}-\tilde{g})|_{\tilde{g}} < \varepsilon
\end{align}
in $\mathcal{D}_{\gamma,\Lambda,s}^{exp} \cap (M\times [0,T]),$ for every $0\leq j\leq k.$  
\end{lem}

\begin{proof}

Suppose that $(M,g(0))\in \mathcal{M}(\delta,\Lambda,s)$ for $\delta \leq \min\{\delta_1,\delta_2\},$ where $\delta_1,\delta_2 >0$ are as in the lemmata above, $\Lambda >0$ and $0< s< \frac{1}{32(\Lambda+1)^2}.$ Let us start by defining the horizontal ends. We first choose $0< \eta \leq \eta_0,$ with $\eta_0$ as in the definition of the model class. If necessary, we rescale the metric so that $\eta \geq \lambda_0, $ where $\lambda_0\geq 2$ is as in Lemma \ref{localstab}. Let $\mathcal{D}_{\gamma,\Lambda,s,\eta}^{out}$ be given by
\begin{align*}
    \mathcal{D}_{\gamma,\Lambda,s,\eta}^{out}= \bigg\{ (x,t)\in M\times [0,\tfrac{1}{32\gamma}]\hspace{0.2cm} ; \hspace{0.2cm} & r_s(x)\leq \tfrac{3}{2}\sqrt{\gamma t +s(\Lambda+1)^2},\\
    &-\frac{\eta}{2}\leq l_s(x)\leq -\frac{\eta}{2}+1 \text{  or  } \frac{\eta}{2}-1 \leq l_s(x)\leq \frac{\eta}{2} \bigg\}.
\end{align*}
Since $\hat{g}$ and $\tilde{g}$ are smooth Riemannian metrics on a compact region, there exists a constant $C<\infty,$ which might initially depend on $s>0,$ such that
\begin{align}\label{paraboundary}
    |\hat{g}-\tilde{g}|_{\tilde{g}} +\sqrt{t+s}|\tilde{\nabla}\hat{g}|_{\tilde{g}}< C
\end{align} 
in $\displaystyle{ \mathcal{D}_{\gamma,\Lambda,s}^{exp}\cap (M\times[0,T])}.$ We define 
\begin{align*}
    \tau_0 = \max\{ \tau \hspace{0.1cm} |\hspace{0.1cm} |\hat{g}-\tilde{g}|_{\tilde{g}} +\sqrt{t+s}|\tilde{\nabla}\hat{g}|_{\tilde{g}} \leq 1 \hspace{0.2cm} \text{in} \hspace{0.2cm}\mathcal{D}_{\gamma,\Lambda,s}^{exp}\cap (M\times[0,\tau]) \}.
\end{align*}
Since $(M,g(0))\in \mathcal{M}(\delta,\Lambda,s)$ we have that $\tau_0>0$ by continuity. Assume $\displaystyle{\tau_0 < \min\{ \tfrac{1}{32\gamma},\tfrac{1}{\lambda_0}\}}.$ We then first show that for every $0\leq j \leq k:$
\begin{align}\label{expandingcontrol}
    (t+s)^{j/2}|(\nabla^{\tilde{g}})^j(\hat{g}-\tilde{g})|_{\tilde{g}} < \varepsilon
\end{align}
in $\mathcal{D}_{\gamma,\Lambda,s,\eta}^{in} \cap (M\times [0,\tau_0]),$ where
\begin{align*}
    \mathcal{D}_{\gamma,\Lambda,s,\eta}^{in}= \bigg\{ (x,t)\in M\times [0,\tfrac{1}{32\gamma}]\hspace{0.2cm} ; \hspace{0.2cm} & r_s(x)\leq \tfrac{3}{2}\sqrt{\gamma t +s(\Lambda+1)^2}, \\
    &-\tfrac{\eta}{2}+1\leq l_s(x)\leq \tfrac{\eta}{2}-1     \bigg\}.
\end{align*}
Note that $\mathcal{D}_{\gamma,\Lambda,s,\eta}^{in}\cup \mathcal{D}_{\gamma,\Lambda,s,\eta}^{out}= \mathcal{D}_{\gamma,\Lambda,s,\eta}^{exp}$
is the expanding region restricted to $-\frac{\eta}{2}\leq l_s(x)\leq \frac{\eta}{2};$ essentially, a neighbourhood of a point in $\Gamma_s= \Phi_s(q_{\max},[-L/2,L/2]).$

Consider $Q=\Phi_s \circ (\tilde{\phi}_s)^{-1},$ where $\tilde{\phi_s}: \expander \rightarrow \expander$ is given by $\tilde{\phi}_s(x,l)=(\phi_N^s(x),\phi_{\mathbb{R}}^s(l)).$ Then
\begin{align*}
    Q^*r_s= 2\sqrt{sf}.
\end{align*}
We define $h(t)=s^{-1}Q^*\hat{g}(st)$ and observe that $\displaystyle{s^{-1}Q^*\tilde{g}(st) = g_0(t)}$ in $\displaystyle{\left\{ \mathbf{r}\leq \frac{1}{2\sqrt{s}}\right\}}.$ Since we are assuming that \eqref{ConStab} is true on $\mathcal{D}^{cone}_{\gamma,\Lambda+1,s}$ and we have that the following equality holds $\displaystyle{\left\{\mathbf{r}\geq \sqrt{\gamma t +(\Lambda+1)^2}\right\}=Q^{-1}\left( \{ \sqrt{\gamma s t +s (\Lambda+1)^2}\leq r_s \leq 3/4 \} \right)},$ we get
\begin{align*}
    |h(t)-g_0(t)|_{g_0(t)} + & \mathbf{r}|\nabla^{g_0(t)}h(t)|_{g_0(t)} = |s^{-1}Q^*\hat{g}(st) - s^{-1}Q^*\tilde{g}(st)|_{g_0(t)}  \\
    & + \mathbf{r}|\nabla^{g_0(t)}h(t)|_{g_0(t)}  \leq Q^*\left(|\hat{g}-\tilde{g}|_{\tilde{g}}  +r_s|\nabla^{\tilde{g}}\hat{g}|_{\tilde{g}}\right)(st)\\
    &< \delta,
\end{align*}
in $\displaystyle{\left\{\mathbf{r}\geq \sqrt{\gamma t +(\Lambda+1)^2}\right\}}$ for any $t\in [0,\tfrac{1}{s}\min \{\tfrac{1}{32\gamma},T\}].$ Furthermore,
\begin{align*}
    |h(0)-g_0(0)|_{g_0(0)} = Q^* (|g(0)-\tilde{g}(0)|) < \delta
\end{align*}
in $\{ \mathbf{r}\leq 2(\Lambda+2)\}$ since $(M,g(0))\in \mathcal{M}(\delta,\Lambda,s).$ Finally, on $\mathcal{D}_{\gamma,\Lambda,s,\eta}^{out}\cap (M\times [0,\tau_0])$ we can use \eqref{paraboundary} to get
\begin{align}\label{ends_assumption}
 \nonumber   |h(t)-g_0(t)|_{g_0(t)}+\sqrt{1+t}|\nabla^{g_0(t)}h(t)|_{g_0(t)} &= |s^{-1}Q^*\hat{g}(st)- s^{-1}Q^*\tilde{g}(st)|_{g_0(t)} \\
    &+ \sqrt{1+t}|\nabla^{g_0(t)}s^{-1}Q^*\hat{g}(st)|_{g_0(t)}\\
 \nonumber   &\leq Q^*\left( |\hat{g}-\tilde{g}|_{\tilde{g}}+ \sqrt{ts+s} |\nabla^{\tilde{g}}\hat{g}|_{\tilde{g}}\right)(st)\\
 \nonumber   & \leq 1
\end{align}
 in $\left(\{\mathbf{r}\leq 3/4\}\times [-\eta/2,-\eta/2+1]\right) \cup \left( \{\mathbf{r}\leq 3/4\}\times [\eta/2-1,\eta/2]\right)$ for $t\in [0,\tfrac{1}{s}\min \{\tfrac{1}{\lambda_0},T\}].$

We can now apply Lemma \ref{localstab} with $M=1.$ Thus, for every $\varepsilon>0,$ there exists $\delta_0(g_N,\varepsilon,k)>0$ such that if $\delta\leq \delta_0,$  
\begin{align*}
    \sup_{D'_{\eta}}|(t\partial_t)^a(\sqrt{t}\nabla^{g_0(t)})^b(h(t)-g_0(t))|_{g_0(t)}< \varepsilon
\end{align*}
for any $0\leq a+b\leq 2k.$ The same follows for $\hat{g}(t)-\tilde{g}(t),$ i.e., 
\begin{align*}
    |(t\partial_t)^a(\sqrt{t}\nabla^{\tilde{g}(t)})^b(\hat{g}(t)-\tilde{g}(t))|_{\tilde{g}(t)}< \varepsilon
\end{align*}
for every $(x,t)\in M\times [0,\tau_0]$ with $r_s(x)\leq \tfrac{3}{2}\sqrt{\gamma t +s(\Lambda+1)^2}$ and $-\frac{\eta}{4}\leq l_s(x)\leq \frac{\eta}{4}.$

We now show these estimates hold along the whole curve. As mentioned before, we will cover the curve with sub-regions as the one above and overlap them so that every point of $\mathcal{D}^{out}$ is always contained in some interior region $\mathcal{D}^{in}.$ Let $n_0\in \mathbb{N}$ be such that $n_0\geq \tfrac{L}{\eta_0}.$ We define the regions
\begin{align*}
    \mathcal{D}_{\gamma,\Lambda,s,2i-1}^{exp}= \bigg\{ &(x,t)\in M\times [0,\tau_0]; \hspace{0.1cm} r_s\leq \tfrac{3}{2} \sqrt{\gamma t +s(\Lambda+1)^2},\\
    & l_s(x)\in \left[(2(i-1)-n_0)\tfrac{L}{2n_0},(2i-n_0)\tfrac{L}{2n_0}\right]  \bigg\}=:\mathcal{D}_{2i-1},
 \end{align*}  
 and
 \begin{align*}
   \mathcal{D}_{\gamma,\Lambda,s,2j}^{exp}= \bigg\{ & (x,t)\in M\times [0,\tau_0];\hspace{0.1cm}  r_s\leq \tfrac{3}{2} \sqrt{\gamma t +s(\Lambda+1)^2},\\
   &l_s(x)\in \left[ (2j-1 -n_0)\tfrac{L}{2n_0},  (2j+1-n_0)\tfrac{L}{2n_0}\right] \bigg\}
   =:\mathcal{D}_{2j},
\end{align*}
for every $i,j\in \{1,2,...,n_0\}.$ Thus,  $\displaystyle{\mathcal{D}_{\gamma,\Lambda,s}^{exp}\cap \left(M\times[0,\tau_0]\right) \subset \bigcup_{k=1}^{2n_0}\mathcal{D}_k.}$ Furthermore, for $k\in \{1,..., 2n_0-1\},$ the following inclusions hold: 
\begin{align}\label{inclusions}
  \nonumber  & \mathcal{D}_{k}^{out}\subset \mathcal{D}_{k-1}^{in}\cup\mathcal{D}_{k+1}^{in},\\
&\mathcal{D}_{1}^{out}\subset\mathcal{D}_{2n_0}^{in}\cup\mathcal{D}_{2}^{in},\\
\nonumber & \mathcal{D}_{2n_0}^{out}\subset\mathcal{D}_{2n_0-1}^{in}\cup\mathcal{D}_{1}^{in},
\end{align}
with $\displaystyle{\mathcal{D}_{k}^{out}}$ and $\displaystyle{\mathcal{D}_{k}^{in}}$ defined as before.

So far, we have obtained estimate \eqref{expandingcontrol}
\begin{align*}
     (t+s)^{j/2}|(\nabla^{\tilde{g}})^j(\hat{g}-\tilde{g})|_{\tilde{g}} < \varepsilon,
\end{align*}
for every $(x,t) \in \displaystyle{\mathcal{D}_{2i-1}^{in}},$ for any $i\in \{0,1,...,n_0\}.$ However, by the inclusions in \eqref{inclusions}, any point $\displaystyle{(x,t)\in \mathcal{D}_{k}^{out}}$ is such that $\displaystyle{(x,t)\in \mathcal{D}_{k-1}^{in}\cup\mathcal{D}_{k+1}^{in}}$ (analogously if $\displaystyle{(x,t)\in {\mathcal{D}_{1}^{out}},\mathcal{D}_{2n_0}^{out}).}$ In particular, after applying the same reasoning finitely many times (but at least $2n_0$ times), we can make sure any point in $\displaystyle{\mathcal{D}_{\gamma,\Lambda,s}^{exp}\cap \left(M\times[0,\tau_0]\right)}$ is inside an interior region $\displaystyle{\mathcal{D}^{in}_k}.$ It follows that assumption \eqref{paraboundary} can be improved to
\begin{align}\label{improved_out}
     (t+s)^{j/2}|(\nabla^{\tilde{g}})^j(\hat{g}-\tilde{g})|_{\tilde{g}}(x) < \varepsilon
\end{align}
for every such $\displaystyle{(x,t)\in \mathcal{D}_{k}^{out}.}$

Therefore, at $t=\tau_0,$ \eqref{paraboundary} (and, thus, \eqref{ends_assumption}) holds with the LHS less or equal to $\varepsilon < 1.$ This is already a contradiction with the definition of $\tau_0$ being the maximal time $t\in [0,T]$ such that \eqref{paraboundary} holds with $C=1.$ Then, it follows that $\displaystyle{\tau_0\geq \min\{\tfrac{1}{32\gamma},\tfrac{1}{\lambda_0},T\}}$ and the estimates in the statement of the lemma hold for $t\in [0, \min\{\tfrac{1}{32\gamma},\tfrac{1}{\lambda_0}\}].$ Without loss of generality, choose $\displaystyle{\tfrac{1}{32\gamma}\leq \tfrac{1}{\lambda_0}}.$

Putting everything together, we conclude that our estimates hold along the whole curve and depend only on the assumptions for the conical region and for the initial data, i.e., our choice of $\delta>0.$ Thus,
\begin{align*}
    (t+s)^{j/2}|(\nabla^{\tilde{g}})^j(\hat{g}-\tilde{g})|_{\tilde{g}} < \varepsilon
\end{align*}
for every $(x,t)\in \mathcal{D}_{\gamma,\Lambda,s}^{exp},$ which finishes the proof of the lemma.

\end{proof}

\subsection{Uniform curvature estimates}
We now put the estimates from the previous section together to prove curvature estimates for a complete Ricci flow starting in the class $\mathcal{M}(\delta,\Lambda, s).$ The proof now is essentially the same as in \cite{conicalsing}, so we refer the reader to it for the details.

\begin{thm} \label{uniCurvature}
Given $\Lambda >0,$ there exist $\delta_0(g_N), s_0(\Lambda), t_0(g_N), C(g_N)$ such that for every $s\in (0,s_0]$ the following holds. If $(M,g(t))_{t\in [0,T]}$ is a complete Ricci flow with bounded curvature and initial data satisfying $(M,g(0))\in \mathcal{M}(\delta_0,\Lambda,s),$ then
\begin{align*}
    &\max_{\{r_s \leq 3/4\}} |Rm(g(t))|_{g(t)}\leq \frac{C}{t}, \hspace{0.2cm} \text{for}\hspace{0.1cm} t \in (0,\min\{t_0,T\}],\\
    & \max_{\{r_s\leq 3/4\}}r_s^{j+2}|(\nabla^{g(t)})^j Rm(g(t))|_{g(t)} \leq C,
\end{align*}
for $t \in (0,\min\{t_0,T\} ].$ Furthermore, for every $\varepsilon >0$ and integer $k\geq 0,$ there exist $\delta_1=\delta_1(\varepsilon,k,g_N)$ and $\gamma_1=\gamma_1(\varepsilon,k,g_N)$ such that if $s\in (0,s_0]$ and $\gamma \geq \gamma_1,$ then we have the following. If $(M,g(0))\in \mathcal{M}(\delta_1,\Lambda,s),$ then for every $t\in (0, \min\{ (32\gamma)^{-1},T\}] $ there is a map 
\begin{align*}
    \Theta_{s,t}: \left\{ r_s \leq \frac{5}{4}\sqrt{\gamma t+ s(\Lambda+1)^2}  \right\} \longrightarrow N\times\mathbb{R},
\end{align*}
a diffeomorphism onto its image, such that 
\begin{align*}
    \{ \mathbf{r}_s\leq \sqrt{\gamma t}\}\times [-L/2,L/2] &\subset Im(\Theta_{s,t})\\
    &\subset \{ \mathbf{r}_s \leq \tfrac{3}{2}\sqrt{\gamma t +s(\Lambda+1)^2}\} \times [-L/2, L/2],
\end{align*} 
and for any non-negative index $j\leq k,$
\begin{align*}
    |(\sqrt{t+s})\nabla^{(g_{\exp}(t+s))})^j ((\Theta_{s,t}^{-1})^* g(t)-{g_{\exp}}(t+s))|_{{g_{\exp}}(t+s)} < \varepsilon 
\end{align*}
in $Im(\Theta_{s,t}),$ where $g_{\exp}(t)=g_N(t)\otimes \gcan(t).$
\end{thm}

\begin{proof}
The curvature estimates follow from Lemma \ref{conelemma} and Lemma \ref{expEstimate}, applying a loop argument to show that the threshold $B$ from Lemma \ref{conelemma} can always be avoided (see \cite{conicalsing} for the details).

The statement for $\Theta_{t,s}$ follows by defining $\Theta_{s,t}=\Phi_s \circ \psi_t$ and combining, again, lemmata \ref{conelemma} and \ref{expEstimate}. The inclusions for $Im(\Theta_{t,s})$ are a direct application of \cite[Theorem 3.1]{conicalsing} and the fact that
\begin{align*}
    |\tilde{\nabla}l_s |_{\tilde{g}}(x,t)=|\nabla^{g_0(t/s)}l|_{g_0(t/s)}(\phi_s(\Phi_s^{-1}(x)))=1,
\end{align*}
and $Im(l_s) = [-L/2,L/2].$
\end{proof}

If, in addition, we assume a bound $|Rm(g(0))| \leq A$ outside the conical and expanding regions, we obtain a global bound for $Rm$ in time as a direct application of Shi's estimates and the maximum principle.

\begin{cor}\label{CurvCorollary}
Let $(M,g)\in \mathcal{M}(\delta_0,\Lambda,s)$ for $s\in (0,s_0],$ where $\delta_0(g_N)$ and $s_0(\Lambda)$ are given by Theorem \ref{uniCurvature}. Suppose that
\begin{align*}
    \sup_{M\backslash Im(\Phi_s)}|Rm(g)|)g \leq A.
\end{align*}
Then, there exists $T(A,g_N)$ and $C(A,g_N)$ such that the Ricci flow $g(t)$ with $g(0)=g$ exists for $t\in [0,T(A,g_N)]$ and satisfies
\begin{align*}
    \max_{M\times (0,T]}|Rm(g(t))|_{g(t)} \leq \frac{C(A,g_N)}{t}.
\end{align*}
Moreover, all the conclusions from Theorem \ref{uniCurvature} hold.

\end{cor}

\section{The Ricci flow out of spaces with edge type conical singularities}\label{theSolution}
In this section, we assume that $r_0=1,$ which is true module a rescaling argument, $\Lambda_1\geq \Lambda_0,$ and that $\kappa >0$ is such that $\gamma_Z(r) < \kappa$ for $r\in (0,1].$
We will approximate $(Z,g_Z)$ by removing small conical neighbourhoods of points in $\Gamma$ and glueing in small neighbourhoods of points of the spine of $\expander$ at scale $s>0.$ After making sure we have a smooth Riemannian manifold, we can flow it by Ricci flow. The estimates from Section \ref{uniformEstimates} will allow us to pass to a limit solution as $s\searrow 0.$ This limit will be shown to converge back to $(Z,g_Z)$ in the pointed Gromov--Hausdorff topology and, away from the singular curve, in a locally smooth sense.

\subsection{The approximating sequence}
Let $s\in (0,1/2]$ and $\{p_1,p_2,...,p_{n_0}\}\subset\Gamma$ be points on $\Gamma$ such that $\displaystyle{\mathbf{l}\left( \Gamma_{|_{[p_k,p_{k+1}]}} \right) =\frac{L}{n_0}=\eta,}$ for every $k \in \{1,2,...,n_0\},$ where $0<\eta\leq \eta_0$ and we consider $p_{n_0+1}=p_1.$ For each $p_k,$ we pick $q_k\in \Gamma$ to be the middle point between $p_k$ and $p_{k+1},$ and consider the map $\phi_k := \phi_{q_k}$ from Definition \eqref{singular_space}, i.e., $\displaystyle{\phi_k: (0,1] \times X \times [-\eta/2,\eta/2]\longrightarrow Z.}$ Thus, $\phi_k$ parametrises a neighbourhood of $q_k\in \Gamma.$ We then define $M_s$ by an iteration method. 

First, let
\begin{align*}
    Z_{s}^k:= Z\backslash \phi_k\left( (0,s^{1/4}]\times X\times [-\eta/2,\eta/2] \right)
    \end{align*}
and
\begin{align*}
    Z_s:= Z\backslash \bigcup_{k=1}^{n_0}\phi_k\left( (0,s^{1/4}]\times X\times [-\eta/2,\eta/2] \right).
\end{align*}
We will work with a subdivision of $I_L:=[-L/2,L/2]$ by sub-intervals of length $\eta.$ We denote $I_{\eta}^{k}:=[-L/2 +(k-1)\eta, -L/2 +k\eta]$ for each $1\leq k \leq n_0.$ Then, we define
\begin{align*}
    M_{s}^1:= \frac{Z_{s}^1 \bigsqcup \left(\{\mathbf{r}_s\leq 1\}\times I_{\eta}^{1} \right) }{\{\phi_1(r,\cdot,l+L/2-\eta/2)=({F}_s(r,\cdot),l), \hspace{0.2cm} r\in[s^{1/4},1], \hspace{0.2cm} l\in I_{\eta}^{1} \}}.
\end{align*}
Using $M_{s}^1$ we can define $M_{s}^2:$
\begin{align*}
    M_{s}^2=\frac{\left(M_{s}^1\backslash \phi_2( (0,s^{1/4}]\times X\times [-\eta/2,\eta/2])\right) \bigsqcup\left( \{\mathbf{r}_s\leq 1\}\times I_{\eta}^{2} \right) }{\thicksim_2},
\end{align*}
where $\thicksim_2$ is given by the following identification: 
\begin{itemize}
   \item $\phi_2(r,v,l+L/2-3\eta/2)=({F}_s(r,v),l),$ for $v\in X,$ $l\in I_{\eta}^{2}$ and $r\in [s^{1/4},1],$
   \item $\phi_2(r,v,-\eta/2)=\phi_1(r,v,\eta/2),$ for $r\in (0,1]$ and $v\in X.$
\end{itemize}
Now, assuming $M_{s}^{k-1}$ is defined for $1<k<n_0$, we define $M_{s}^k$ as 
\begin{align*}
    M_{s}^k=\frac{\left(M_{s}^{k-1}\backslash \phi_k( (0,s^{1/4}]\times X\times[-\eta/2,\eta/2])\right) \bigsqcup \left(\{\mathbf{r}_s\leq 1\}\times I_{\eta}^{k}\right) }{\thicksim_k},
\end{align*}
where $\thicksim_k$ is given by:
\begin{itemize}
    \item $\phi_k(r,v,l+L/2-(2k-1)\eta/2)=({F}_s(r,v),l),$ for $v\in X,$ $l\in I_{\eta}^{k}$ and $r\in [s^{1/4},1],$
    \item $\phi_k(r,v,-\eta/2)=\phi_{k-1}(r,v,\eta/2),$ for $r\in (0,1]$ and $v\in X.$
\end{itemize}

Finally, we define $M_s$ as follows. 
\begin{align*}
    M_{s}=\frac{\left(M_{s}^{n_0-1}\backslash \phi_{n_0}( (0,s^{1/4}]\times X\times [-\eta/2,\eta/2])\right) \bigsqcup\left( \{\mathbf{r}_s\leq 1\}\times [L/2-\eta,L/2]\right) }{\thicksim_{n_0}},
\end{align*}
where $\thicksim_{n_0}$ is given by:
\begin{itemize}
    \item $\phi_{n_0}(r,v,l-L/2 +\eta/2)=({F}_s(r,v),l),$ for $v\in X,$ $l\in[L/2-\eta,L/2]$ and $r\in [s^{1/4},1],$
    \item $\phi_{n_0}(r,v,-\eta/2)=\phi_{n_0-1}(r,v,\eta/2),$ for $r\in (0,1]$ and $v\in X,$
    \item $\phi_{n_0}(r,v,\eta/2)=\phi_{1}(r,v,-\eta/2),$ for $r\in (0,1]$ and $v\in X.$
\end{itemize}

\begin{rem}
We are picking enough points in $\Gamma$ so that each small neighbourhood of these points is isometric to part of $C(X)\times\mathbb{R}$ and they cover the curve. By removing this neighbourhood and glueing in part of $N\times\mathbb{R},$ we smooth out the curve. The equivalence relations above make sure our final manifold is smooth. 
\end{rem}

Given the construction above, we can consider the natural embeddings 
\begin{align*}
    \Phi_s : \{ \mathbf{r}_s\leq 1 \} \times [-L/2,L/2]\longrightarrow M_s \hspace{0.2cm} \text{and} \hspace{0.2cm}  \Psi_s : Z_s \longrightarrow M_s,
\end{align*}
under the identification $\displaystyle{\Phi_s(\cdot,-L/2)=\Phi_s(\cdot,L/2)}.$ We also define the functions $r_s$ and $l_s$ by
\begin{align*}
   r_s(x)= \begin{cases}
          \Lambda_1\sqrt{s},\hspace{0.2cm} \text{if } x\in \Phi_s(\{\mathbf{r}_s\leq \Lambda_1\sqrt{s}\}\times [-L/2,L/2]),\\
           (\pi_1\circ \Phi_s^{-1})^*\mathbf{r}_s(x),\hspace{0.2cm} \text{if } x\in \Phi_s(\{ \Lambda_1 \sqrt{s}\leq \mathbf{r}_s\leq 1\}\times[-L/2,L/2]),\\
           1,\hspace{0.2cm} \text{if } x \in M_s\backslash Im(\Phi_s)
\end{cases}
\end{align*}
and 
\begin{align*}
     l_s(x):=\begin{cases}
            (\pi_2 \circ \Phi_s^{-1})(x) \hspace{0.2cm}\text{if } x \in Im(\Phi_s),\\
            L/2, \hspace{0.2cm} \text{if } x \in M_s\backslash Im(\Phi_s). 
\end{cases}
\end{align*}
This allows us to define $\displaystyle{\Gamma_s:= \Phi_s\left(\{q_{\max}\}\times [-L/2,L/2]\right)}.$ Furthermore, we define our parametrisation of the conical region, $\displaystyle{ \phi: [s^{1/4}, 1]\times X \times [-L/2,L/2]\rightarrow M_s},$ as follows. If $l\in[-L/2+(k-1)\eta,-L/2+ k\eta],$ let $\displaystyle{\phi(\cdot,\cdot,l)=\phi_k(r,v,l+L/2-(2k-1)\eta/2)}.$ In particular, we have that $\displaystyle{ r_s= \left( \left( \Psi_s \circ \phi\right)^{-1}\right)^*r}$ and $\displaystyle{l_s= \left( \left( \Psi_s \circ \phi\right)^{-1}\right)^*l}$ in $Im(\Phi_s)\cap Im(\Psi_s).$

To define the metric on $M_s,$ we consider the following regions on our manifold:
\begin{align*}
    U_{2i-1}:= \left\{ x \in M; r_s(x) < 1 \hspace{0.2cm} \text{and} \hspace{0.2cm} l_s(x) \in \left[\frac{2(i-1)-n_0}{2n_0}L, \frac{2i-n_0}{2n_0}L\right) \right\}
\end{align*}
and
\begin{align*}
    U_{2j}:= \left\{ x \in M; r_s(x) < 1 \hspace{0.2cm} \text{and} \hspace{0.2cm} l_s(x) \in \left[\frac{2j-1-n_0}{2n_0}L, \frac{2j+1-n_0}{2n_0}L\right) \right\},
\end{align*}
for $i,j\in \{ 1,2,...,n_0\},$ where on $U_{2n_0},$ we are identifying
\begin{align*}
    \left[\frac{n_0-1}{2n_0}L, \frac{n_0 +1}{2n_0}L\right) \cong \left[\frac{n_0-1}{2n_0}L, L/2\right)\cup \left[\frac{-L}{2}, \frac{1-n_0}{2n_0}L \right).
\end{align*}
We consider $\displaystyle{\{f_k\}_{k=1}^{2n_0}}$ a differentiable partition of unit subordinated to ${\{U_k\}_{k=1}^{2n_0}}.$ Let also $\Phi_{s,k}$ be the restriction of $\Phi_s$ to $\Phi_s^{-1}(U_k)$ and let $g_{\exp,k}(s)$ be the metric induced by $g_{\exp}(s)$ on $\Phi_s^{-1}(U_k).$ We then define
\begin{align*}
    g_s=\xi_3\left(\tfrac{r_s}{s^{1/4}}\right)\left( \sum_{k=1}^{2n_0}f_k(\Phi_{s,k}^{-1})^*g_{\exp,k}(s) \right) + \left(1-\xi_3\left(\tfrac{r_s}{s^{1/4}}\right)\right)(\Psi_s^{-1})^*g_Z,
\end{align*}
where $\xi_3$ is a smooth, positive and non-increasing function equal to 1 in $(-\infty,1]$ and 0 in $[2,+\infty).$  Therefore,
\begin{align*}
    g_s= \begin{cases}
      (\Psi_s^{-1})^*g_Z \hspace{0.2cm} \text{in} \hspace{0.2cm} \{ r_s \geq 2 s^{1/4}\},\\
      \sum_{k=1}^{2n_0}f_k(\Phi_{s,k}^{-1})^*g_{\exp,k}(s) \hspace{0.2cm} \text{in} \hspace{0.2cm} \{ r_s \leq s^{1/4}\}.
    \end{cases}
\end{align*}

Let $A< \infty$ be such that $\displaystyle{ \max_{r_s=1}|Rm(g_s)|_{g_s} \leq A,}$ which exists because of the definition of $g_s.$ We then consider $\delta_0=\delta_0(g_N)>0$ as in Theorem \ref{uniCurvature}. Choosing $\kappa$ small and $\Lambda_1$ large, we get $(M_s,g_s)\in \mathcal{M}(\delta_0,\Lambda_1,s),$ where $\eta_0=\eta_0(g_Z).$ 
 
\subsection{The limit solution}

We now consider Ricci flows $g_s(t)$ on $M_s,$ with $g_s(0)=g_s$ and $t\in [0,T],$ where $T>0$ is given by Corollary \eqref{CurvCorollary}. It follows from the same corollary that 
\begin{align}\label{curvatureDecay}
    \max_{M_s}|Rm(g_s(t))|_{g_s(t)} \leq \frac{C_M}{t}
\end{align}
for $t\in (0,T],$ and 
\begin{align}\label{radialDecay}
    \max_{M_s}\sum_{j=0}^{2} r_s^{j+2}|(\nabla^{g_s(t)})^j Rm(g_s(t))|_{g_s(t)} \leq C_M,
\end{align}
for $t\in [0,T].$ In particular, the second estimate yields $\displaystyle{Vol_{g_s(t)}(B_{g_s(t)}(x,t))\geq v_0,}$ for $t\in [0,T]$ and some $x\in \{r_s=1\}.$

Using Hamilton's compactness theorem for Ricci flows, we consider a sequence $s_l \searrow 0$ and, up to a subsequence, there exists a smooth, compact Ricci flow $(M,g(t))_{t\in (0,T]}$ such that 
\begin{align*}
    (M_{s_l},g_{s_l}(t))_{t\in (0,T]} \rightarrow (M,g(t))_{t\in (0,T]},
\end{align*}
where the convergence is given by the existence of diffeomorphisms $H_l:M_{s_l}\rightarrow M$ such that 
\begin{align}\label{RFconvergence}
    H_l^{*}g_{s_l}(t) \rightarrow g(t)
\end{align}
uniformly locally in $M\times (0,T]$ in the $C^{\infty}-$topology. By defining $\displaystyle{\tilde{\Psi}_{s_l}=H^{-1}_{s_l}\circ \Psi_{s_l}:Z_{s_l}\rightarrow M},$ we can show that there is a map $\Psi:Z\backslash \Gamma \rightarrow M,$ diffeomorphism onto its image, such that $\tilde{\Psi}_{s_l} \to \Psi$ in $C^{\infty}$ uniformly away from $\Gamma.$ This follows directly from the curvature bounds for $g_s(t)$ and Arzela-Ascoli (for an example of such construction, see \cite{conicalsing}).

\subsection{Properties of the limit solution}
From \eqref{curvatureDecay}, it follows that $g(t)$ satisfies
\begin{align}\label{curvDecayLimit}
    |Rm(g(t))|_{g(t)}\leq \frac{C_M}{t}
\end{align}
on $M\times (0,T].$ Moreover, given the smooth convergence of $\tilde{\Psi}_{s_l}$ to $\Psi,$ as well as their respective inverses, the identity $H_{s_l}^*r_{s_l}=(\Psi_{s_l}^{-1}\circ H_{s_l})^* (\phi^{-1})^*r$ in $Im(H_{s_l}^{-1}\circ \Psi_{s_l})$ implies that $\displaystyle{ H_{s_l}^{*}r_{s_l} \to (\Psi^{-1})^* ((\phi^{-1})^*r),}$ locally uniformly on $Im(\Psi)$ in the $C^{\infty}$ sense. We then define the following function on $M$:
\begin{align*}
    r_M = \begin{cases}
          (\Psi^{-1})^* ((\phi^{-1})^*r) \hspace{0.2cm} \text{in} \hspace{0.2cm} Im(\Psi\circ\phi),\\
          0 \hspace{0.2cm} \text{in} \hspace{0.2cm} (Im(\Psi))^c,\\
          1 \hspace{0.2cm} \text{otherwise}.
      \end{cases}
\end{align*}
The convergence for $H_{s_l}^*r_{s_l}$ and estimate \eqref{radialDecay} yield
\begin{align}\label{LimitRadialDecay}
    \sum_{j=0}^{2} r_M^{j+2}|(\nabla^{g(t)})^j Rm(g(t))|_{g(t)}\leq C_M
\end{align}
on $M\times (0,T].$ It is also straightforward to see that $\Psi^*g(t) \to g_Z$ as $t\to 0$ in $C^{\infty}_{loc},$ away from the singular curve $\Gamma.$

\subsection{Improvement of expanding estimates for small scales}

We show that if $s>0$ gets even smaller, the closeness of our solution to the expander improves.

\begin{lem}\label{SmallScales}
For every $\varepsilon>0$ and $k\in \mathbb{Z}_+,$ there exist $\eta_1(\varepsilon,k)>0,$ $s_2(\varepsilon,k)>0$ small and $\gamma_3(\varepsilon,k),$ $\Lambda_2(\varepsilon,k)$ large such that the following holds. For each $s\in (0,s_2],$ $\gamma\geq \gamma_3$ and $t\in (0,\eta_1(32\gamma)^{-1}],$ there is a map 
\begin{align*}
    \Theta_{s,t}:\left\{r_s\leq \frac{5}{4}\sqrt{\gamma t +s(\Lambda_2+1)^2} \right\} \longrightarrow \expander,
\end{align*}
diffeomorphism onto its image, such that $\forall\hspace{0.1cm} 0\leq j \leq k,$ 
\begin{align*}
    (t+s)^{j/2}\left| (\nabla^{g_{\exp}(t+s)})^j \left[(\Theta_{s,t}^{-1})^*g_s(t)-g_0(t+s)\right]\right|_{g_{\exp}(t+s)}< \varepsilon
\end{align*}
in $Im(\Theta_{s,t})$ and
\begin{align*}
    \left\{\mathbf{r}_s\leq \sqrt{\gamma t} \right\}\times [-L/2,L/2]& \subset Im(\Theta_{s,t})\\
    &\subset \left\{\mathbf{r}_s\leq \tfrac{3}{2}\sqrt{\gamma t +s(\Lambda_2+1)^2}\right\}\times [-L/2,L/2].
\end{align*}
\end{lem}

\begin{proof}
The proof follows \cite[Lemma 5.1]{conicalsing}. The idea is to show that the rescaled solution $\displaystyle{\tfrac{1}{\eta_1}g_s}$ satisfies $\displaystyle{(M_s,\frac{1}{\eta_1}g_s)\in \mathcal{M}(\delta_1,\Lambda_2,s/\eta_1)}$ for any $s\in (0,s_2],$ with map $\displaystyle{\Phi_{s/\eta_1}},$ and functions $\displaystyle{ r_{s/\eta_1}=\max \left\{ \Lambda_2\sqrt{\tfrac{s}{\eta_1}}, \min\left\{1,\tfrac{r_s}{\sqrt{\eta_1}}\right\}\right\}}$ and $\displaystyle{ l_{s/\eta_1}=\pi_2\circ \Phi_{s/\eta_1}^{-1}.}$ We only observe that 
\begin{align*}
    l_{s/\eta_1}=\pi_2\circ \tilde{\phi}_{1/\eta_1}^{-1}\circ \Phi_s^{-1}
    =\pi_2\circ \tilde{\phi}_{\eta_1}\circ \Phi_{s}^{-1} = \frac{1}{\sqrt{\eta_1}}l_s,
\end{align*}
where we used that $\tilde{\phi}_t\circ\tilde{\phi}_{t^{-1}}=id_{\expander}$ and $\phi^{\mathbb{R}}_{t}=\frac{1}{\sqrt{t}}id_{\mathbb{R}}.$ Thus, $\displaystyle{l_{s/\eta_1}\in \left[\tfrac{-L}{2\sqrt{\eta_1}},\tfrac{L}{2\sqrt{\eta_1}}\right]}$ and is equivalent to $\displaystyle{l_s\in \left[\tfrac{-L}{2},\tfrac{L}{2}\right]}.$ This guarantees the extra direction creates no problem for us, and the rest of the proof follows the analogous result in \cite{conicalsing}.
\end{proof}

\subsection{Distance control for high curvature region}

In this subsection, we prove curvature bounds for points away from $\displaystyle{\{r_M=0\}}.$ Furthermore, for points of high curvature, we show that the distance from $\displaystyle{\{r_M=0\}}$ must be bounded on each time slice. 
\begin{lem}\label{distCurvControl}
There exists $c_0>0$ with the following property: for small $\zeta >0$ there exists $C_{\zeta} >0$ such that if $t\in (0,c_0\zeta],$ then
\begin{align*}
    &d_{g(t)}\left( \{r_M=\sqrt{\gamma t}\}, \{r_M=0\}\right) \leq C_{\zeta}\sqrt{t},\\
    & |Rm(g(t)|_{g(t)}\leq \frac{\zeta}{t} \hspace{0.2cm} \text{in} \hspace{0.2cm} \{r_M > \sqrt{\gamma t}\},
\end{align*}
where $C_{\zeta}=c(g_N)\sqrt{\frac{C_M}{\zeta}}$ and $\gamma=C_M\zeta^{-1}.$
\end{lem}
\begin{proof}
Fix $\varepsilon=10^{-2}.$ By \eqref{LimitRadialDecay}, with $k=0$ in Lemma \ref{SmallScales}, there are $\Lambda_2$ and $\eta_1$ such that if $\gamma=\frac{C_M}{\zeta}$ and $\zeta$ is small enough, then for small enough $s_l>0$ and each $t\in (0,\frac{\eta_1}{32\gamma}],$ there is a map $\displaystyle{ Q_{s_l,t}: \left\{r_{s_l}\leq \frac{5}{4}\sqrt{\gamma t + s_l(\Lambda_2+1)^2}\right\}\longrightarrow \expander}$ satisfying
\begin{align*}
    |(Q_{s_l,t}^{-1})^*g_{s_l}(t)-g_{\exp}(t+s_l)|_{g_{\exp}(t+s_l)}< \frac{1}{100},
\end{align*}
in $Im(Q_{s_l,t})\subset \left \{ \mathbf{r}_s \leq \tfrac{3}{2}\sqrt{\gamma t + s_l(\Lambda_2+1)^2}\right\}\times [-L/2,L/2]$ and
\begin{align*}
    |Rm(g(t))|_{g(t)} \leq \frac{C_M}{r_M^2}< \frac{C_M}{\gamma t}=\frac{\zeta}{t}
\end{align*}
in $\{r_M > \sqrt{\gamma t}\},$ for $t\in (0, \eta_1(32\gamma)^{-1}].$

The distance from $\displaystyle{\left\{r_{s_l}=\sqrt{\gamma t + s_l(\Lambda_2+1)^2}\right\}}$ to $\displaystyle{\left\{ r_{s_l}=\Lambda_2\sqrt{s_l}\right\}}$ is bounded as follows. Just for simplicity, let $\displaystyle{\left\{ r_{s_l}=\Lambda_2\sqrt{s_l}\right\}=:U^{0}_{s_l,\Lambda_2}}$ and $\displaystyle{\left\{r_{s_l}=\sqrt{\gamma t + s_l(\Lambda_2+1)^2}\right\}:= U^{t}_{s_l,{\Lambda_2+1}}}.$ The second part of Theorem \ref{uniCurvature} and Lemma \ref{SmallScales} together yield
\begin{align*}
    &d_{g_{s_l}(t)} \left( U^{t}_{s_l,{\Lambda_2+1}}, U^{0}_{s_l,\Lambda_2}\right)\leq d_{(Q_{s_l,t}^{-1})^*g_{s_l}(t)}\left(Q_{s_l,t}\left(U^{t}_{s_l,{\Lambda_2+1}}\right),Q_{s_l,t}\left(U^{0}_{s_l,\Lambda_2}\right)\right)\\
    & \leq \sqrt{1.01}d_{g_0(t+s_l)}\left( \pi_1\left( Q_{s_l,t}\left(U^{t}_{s_l,{\Lambda_2+1}}\right)  \right)\times I_L,\pi_1\left(Q_{s_l,t}\left(U^{0}_{s_l,\Lambda_2}\right)  \right)\times I_L  \right),
\end{align*}
where $\displaystyle{I_L=[-L/2,L/2]}.$ Working on $(N,g_N(t+s_l)),$ \cite[Lemma 5.3]{conicalsing} yields
\begin{align*}
    &d_{g_{s_l}(t)} \left( U^{t}_{s_l,{\Lambda_2+1}}, U^{0}_{s_l,\Lambda_2}\right)\\
    &\leq \sqrt{1.01}\text{diam}_{g_N(t+s_l)}\left( \left\{ \mathbf{r}_s\leq \frac{3}{4}\sqrt{\gamma t +s_l(\Lambda_2+1)^2}\right\} \right)\leq C_{\zeta}\sqrt{t+s_l}.
\end{align*}
Since $\displaystyle{H_l^*r_{s_l}=H_l^*(\Psi_{s_l}^{-1})^*(\phi^{-1})^* r}$ in the region above, $H_l^{*}g_{s_l}(t) \rightarrow g(t)$ and $\Psi_{s_l}^{-1}\circ H_l \rightarrow \Psi^{-1}$ give
\begin{align*}
    d_{g(t)}\left( \left\{r_M = \sqrt{\gamma t}\right\}, \left\{ r_M=0\right\} \right) \leq C_{\zeta}\sqrt{t},
\end{align*}
by definition of $r_M.$
\end{proof}

\subsection{Gromov--Hausdorff convergence to the initial data}

In this section, we prove that ${(M,d_{g(t)}) \to (Z,d_Z)}$ in the Gromov--Hausdorff sense as $t\searrow0.$ We show that for every $\varepsilon >0,$ ${\Psi: (Z\backslash \Gamma,d_Z) \rightarrow (M,g(t)) }$ is an $\varepsilon-$isometry for small enough $t>0.$ The result follows from the two lemmata below.
\begin{lem}
For every $\varepsilon>0,$ there exist $\delta_1>0$ and $t_1>0$ such that for all $t\in (0,t_1],$  the map $\displaystyle{\Psi: \{ r\geq \delta_1\} \longrightarrow \{ r_M \geq \delta_1\}}$ satisfies
\begin{align}
    \sup \left\{ |d_{g(t)}(\Psi(z_1),\Psi(z_2))-d_Z(z_1,z_2)|, z_1,z_2 \in \{r \geq \delta_1\} \right\} < 3\varepsilon.
\end{align}
\end{lem}
\begin{proof}
Let $\delta_1>0$ be such that $\text{diam}_Z(\{r=\delta_1\})< \varepsilon+L.$ Since $\Psi^*g(t) \to g_Z$ locally uniformly away from $\Gamma$ as $t\searrow 0,$ it follows that $\text{diam}_{g(t)}(\{r_M= \delta_1\})<\varepsilon+L$ for small $t>0.$ Let $d_{Z,\delta_1}$ be the intrinsic metric on $\{r\geq \delta_1\}$ induced by $g_Z,$ and similarly $d_{g(t),\delta_1}$ the intrinsic metric on $\{r_M\geq \delta_1\}$ induced by $g(t).$ We claim that
\begin{align}\label{intrinsicEstimate}
    |d_{Z,\delta_1}(z_1,z_2)-d_Z(z_1,z_2)| < \varepsilon,
\end{align}
and 
\begin{align}\label{metricEstimate}
    |d_{g(t),\delta_1}(\Psi(z_1),\Psi(z_2))-d_{g(t)}(\Psi(z_1),\Psi(z_2))|<\varepsilon,
\end{align}
for all $z_1,z_2\in \{r\geq \delta_1\}.$

We start by showing \eqref{intrinsicEstimate}. It is clear that $d_Z(z_1,z_2)\leq d_{Z,\delta_1}(z_1,z_2).$ Let us assume that $d_Z(z_1,z_2)<d_{Z,\delta_1}(z_1,z_2).$ Then there exists a path connecting $z_1$ to $z_2,$ starting from $z_1$ and escaping $\{r\geq \delta_1\}.$ Let $q_1\in \{r=\delta_1\}$ be the first point this path touches $\{r=\delta_1\}$ and $q_2\in \{r=\delta_1\}$ be the point where the path leaves $\{r\leq \delta_1\}.$ We can choose $\delta_1 >0$ small enough so that $\displaystyle{d_{Z,\delta_1}(q_1,q_2)\leq d_Z(q_1,q_2) +\varepsilon,}$ since $q_1,q_2 \in \{r=\delta_1\}.$ Therefore,
\begin{align*}
d_Z(z_1,z_2) & < d_{Z,{\delta_1}}(z_1,z_2) \leq d_Z(z_1,q_1)+ d_Z(z_2,q_2) + d_{Z,\delta_1}(q_1,q_2)\\ &\leq d_Z(z_1,q_1)+ d_Z(z_2,q_2) + d_{Z}(q_1,q_2) + \varepsilon\\
    & = d_Z(z_1,z_2)+ \varepsilon,
\end{align*}
which suffices to prove the first estimate. Similarly, we can prove \eqref{metricEstimate}.

By the uniform convergence of $\Psi^*g(t)$ to $g_Z$ away from $\Gamma$ as $t\to 0,$ we also have that $\displaystyle{|d_{Z,\delta_1}(z_1,z_2)-d_{g(t),\delta_1}(\Psi(z_1),\Psi(z_2))|< \ \varepsilon,}$ for all $z_1,z_2\in \{r\geq \delta_1\}.$ Putting everything together and using the triangle inequality, we prove the lemma. Possibly after making $\delta_1>0$ smaller, we can also get $\text{diam}(\{r\leq \delta_1\})< L+\varepsilon.$

\end{proof}

\begin{lem}
For every $\varepsilon>0,$ there exists small enough $t_2>0$ such that the following holds. For any $x\in \{r_M=0\},$ we have
\begin{align*}
    d_{g(t)}\left( x, Im(\Psi)\right) < \varepsilon,
\end{align*}
for all $t\in (0,t_2].$
\end{lem}
\begin{proof}
Take $x\in \{r_M=0\}=Im(\Psi)^c.$ Then, of course, $x\in \{r_M\leq \delta_2\}$ for any $\delta_2>0.$ If $\delta_2$ is small enough, and for large enough $l,$ we consider $q_l\in M_{s_l}$ such that $H_l(x)=q_l.$ Identifying these points allows us to consider the distance $\displaystyle{ d_{g_{s_l}(t)}\left(q_l, \left\{r_{s_l}=\sqrt{\gamma t + s_l(\Lambda_2+1)^2}\right\}\right)}$ as $\displaystyle{d_{H_l^*g_{s_l}(t)} \left(x, \left\{H_l^*(\Psi_{s_l}^{-1})^*(\phi^{-1})^* r=\sqrt{\gamma t + s_l(\Lambda_2+1)^2}\right\}\right).}$

As in Lemma \ref{distCurvControl}, taking the limit we have
\begin{align*}
    d_{g(t)}(x,\{r_M=\delta_2\}) < \varepsilon,
\end{align*}
for small enough $t$ and $\delta_2,$ that will depend on $c_0>0$ from Lemma \ref{distCurvControl}. Therefore, for any $x\in Im(\Psi)^c,$ there exists $y\in \{r_M=\delta_2\}\subset Im(\Psi)$ such that $d_{g(t)}(x,y)< \varepsilon.$

\end{proof}

\subsection{Tangent flow at an edge point}

Let $t_k \searrow 0$ be a sequence of times going to 0. From the convergence in \eqref{RFconvergence}, it follows that there exist a sequence $s_{l_k}\searrow 0$ such that for any $j\leq k,$
\begin{align*}
    (t_k)^{j/2}|(\nabla^g)^j(g-H_{l_k}^*g_{s_{l_k}})|_g (t_k) < \frac{1}{k}
\end{align*}
and $\frac{s_{l_k}}{t_k}\to 0.$ Let $\gamma_k=\gamma_3(1/k,k),$ $\Lambda_k=\Lambda_2(1/k,k)$ and $\eta_k=\eta_1(1/k,k)$ be as given by Lemma \ref{SmallScales} and set $\displaystyle{\tau_k=\frac{\eta_k}{32\gamma_k}}.$ Passing to a subsequence if necessary, we may assume that $t_k < \tau_k$ and $s_{l_k}< s_2(1/k,k).$ Lemma \ref{SmallScales} implies that there exist
\begin{align*}
    \Theta_{k}: \left\{ r_{s_{l_{k}}}\leq \sqrt{\gamma_k t_k+s_{l_k}(\Lambda_k+1)^2}\right\}\longrightarrow \expander,
\end{align*}
diffeomorphism onto their image, such that for $j\leq k,$
\begin{align*}
    (t_k+s_{l_k})^{j/2}|(\nabla^{g_{\exp}(t_k+s_{l_k})})^j((\Theta_{k}^{-1})^*g_{s_{l_k}}(t_k)-g_{\exp}(t_k+s_{l_k}))|_{g_{\exp}(t_k+s_{l_k})} < \frac{1}{k}
\end{align*}
in $Im(\Theta_{k}).$ Set $R_k=(\Theta_{k}\circ H_{s_{l_k}})^{-1}$ so that
\begin{align*}
    (t_k)^{j/2}|(\nabla^{g_{\exp}(t_k+s_{l_k})})^j(R_k^*g(t_k)-g_{\exp}(t_k+s_{l_k}))|_{g_{\exp}(t_k+s_{l_k})} < \frac{C}{k}
\end{align*}
in $Im(\Theta_{k})$ for large $k.$ Using the explicit form of $g_{\exp}(t_k+s_k),$ we get:
\begin{align*}
    \left|(\nabla^{g_N\otimes \gcan})^j\left(\left(R_k\circ \tilde{\phi}^{-1}_{t_k+s_{l_k}}\right)^*\tfrac{1}{t_k}g(t_k)-\left(1+\frac{s_{l_k}}{t_k}\right)g_N\otimes\gcan\right)\right|_{g_N\otimes\gcan}< \frac{C}{k}
\end{align*}
in $\tilde{\phi}_{t_k+s_{l_k}}(Im(\Theta_{k})).$ Set $h_k=(R_k\circ \tilde{\phi}^{-1}_{t_k+s_{l_k}})^*\frac{1}{t_k}g(t_k).$ Thus,
\begin{align}\label{metricConvergence}
    \left| (\nabla^{g_N\otimes\gcan})^j\left(h_k-\left(1+\frac{s_{l_k}}{t_k}\right)g_N\otimes\gcan\right)\right|_{g_N\otimes\gcan}< \frac{C}{k}
\end{align}
in $Im(\tilde{\phi}_{t_k+s_{l_k}}\circ R_k^{-1})= \tilde{\phi}_{t_k+s_{l_k}}(Im(\Theta_{k})).$ 

By Lemma \ref{SmallScales}, we have
\begin{align*}
    \left\{ \mathbf{r}_{s_{l_k}}\leq \frac{1}{2}\sqrt{\gamma_k t_k+s_{l_k}(\Lambda_k+1)^2}\right\} \times I_L\subset \{\mathbf{r}_{s_{l_k}}\leq \sqrt{\gamma_kt_k}\}\times I_L\subset Im(\Theta_{k}),
\end{align*}
where in the first inclusion we assumed that $\displaystyle{\gamma_3(1/k,k)\geq (\Lambda_2(1/k,k)+1)^2}$ without loss of generality.  It follows that 
\begin{align}\label{SetsControl}
   \nonumber \tilde{\phi}& _{t_k+s_{l_k}}(Im(\Theta_{k}))\supset \tilde{\phi}_{t_k+s_{l_k}}\left(  \bigg\{ \mathbf{r}_{s_{l_k}}\leq \tfrac{1}{2}\sqrt{\gamma_k t_k+s_{l_k}(\Lambda_k+1)^2}\bigg\} \times I_L\right)\\
    &= \phi^N_{t_k+s_{l_k}}\left(  \bigg\{ \mathbf{r}_{s_{l_k}}\leq \tfrac{1}{2}\sqrt{\gamma_k t_k+s_{l_k}(\Lambda_k+1)^2}\bigg\}\right) \times \left[-\tfrac{L}{2\sqrt{t_k+s_{l_k}}},\tfrac{L}{2\sqrt{t_k+s_{l_k}}}\right]\\
  \nonumber  &\supset \{\mathbf{r}\leq \sqrt{\gamma_k/8}\}\times I_L,
\end{align}
for large $k.$ The last inclusion follows from \cite[Lemma 5.3]{conicalsing}. Now let $q_k\in M$ be such that $(q_{\max},0)=\tilde{\phi}_{t_k+s_{l_k}}\circ R_k^{-1}(q_k)\in \expander$ satisfies
\begin{align*}
    |Rm(g_N)|_{g_N}(q_{\max})=\max_{N}|Rm(g_N)|_{g_N}.
\end{align*}
Lemma \ref{distCurvControl} for $\zeta=\frac{1}{2}\max_N |Rm(g_N)|_{g_N}$
yields $\hat{C},\hat{\gamma}>1$ such that 
\begin{align*}
    q_k\in \{r_M=\sqrt{\hat{\gamma}t_k}\},
\end{align*}
and 
\begin{align*}
    d_{g(t_k)}\left(\{r_M=\sqrt{\hat{\gamma}t_k}\},\{r_M=0\}\right)\leq \hat{C}\sqrt{t_k}.
\end{align*}
Given any $p_k\notin Im(\Psi),$ we know that $r_M(p_k)=0,$ hence
\begin{align*}
    d_{g(t_k)}(p_k,q_k)\leq \hat{C}\sqrt{t_k}.
\end{align*}
Therefore, 
\begin{align*}
    d_{g_N\otimes\gcan}\left( (q_{\max},0),\tilde{\phi}_{t_k+s_{l_k}}\circ R^{-1}_k(p_k)\right)\leq 2\hat{C},
\end{align*}
for large $k.$ Putting this together with \eqref{metricConvergence}, \eqref{SetsControl} and letting $\gamma_k\to +\infty,$ we get the convergence of ${\left(M,{t_k}^{-1}g(t_k),p_k\right)}$ to ${(\expander,g_N\otimes\gcan,(\Bar{q},\Bar{l}))}$ in the smooth pointed Cheeger-Gromov topology. This is enough to conclude that 
\begin{align*}
    (M,\tfrac{1}{t_k}g(t_k t),p_k)_{t\in (0,t_k^{-1}T]} \rightarrow   {(\expander,h(t),(\Bar{q},\Bar{l}))_{t\in (0,\infty)}}
\end{align*}
in the smooth pointed Cheeger-Gromov topology, where $(N,h(t))$ is a complete Ricci flow with bounded curvature and $h(1)=g_N\otimes\gcan.$ The forward and backward uniqueness of Ricci flows gives $h(t)=g_{\exp}(t)$ for all $t\in (0,\infty),$ and this finishes the proof of Theorem \ref{mainTheorem}. 

\printbibliography
\end{document}